\documentclass[12pt,leqno]
{amsart}

\usepackage[margin=1.2in]{geometry}

\usepackage[colorlinks, linkcolor=blue,citecolor=blue]{hyperref}

\usepackage{amsmath,epsfig,graphicx,color, float}
\usepackage{subcaption}
\usepackage{relsize}
\usepackage{comment}
\usepackage{enumitem}
\usepackage{natbib}

\usepackage{appendix}

\numberwithin{table}{section}
\numberwithin{figure}{section}
\numberwithin{equation}{section}

\definecolor{darkblue}{rgb}{.2, 0.2,.8}
\definecolor{darkgreen}{rgb}{0,0.5,0.3}
\definecolor{darkred}{rgb}{.8, .1,.1}

\newcommand{\bfx}{\vect{x}}
\newcommand{\bfX}{\vect{X}}

\newcommand{\bfalp}{\vect{\alpha}}

\newcommand{\bfpi}{\vect{\pi}}
\newcommand{\bfT}{\mat{T}}
\newcommand{\bft}{\vect{t}}
\newcommand{\bfe}{\vect{e}}

\newcommand{\1}{{\mathbf 1}}
\newcommand{\0}{\mat{0}}

\newcommand{\R}{\mathbb{R}}
\newcommand{\E}{\mathbb{E}}
\renewcommand{\P }{{\mathbb P}}

\newcommand{\ov}{\overline}

\newtheorem{lemma}{Lemma}[section]

\newtheorem{theorem}[lemma]{Theorem}
\newtheorem{proposition}[lemma]{Proposition}
\newtheorem{definition}[lemma]{Definition}
\newtheorem{corollary}[lemma]{Corollary}
\newtheorem{example}{Example}[section]

\newtheorem{remark}{Remark}[section]
\newtheorem{algorithm}{Algorithm}[section]

\usepackage{bm}
\newcommand{\vect}[1]{\pmb{#1}}
\newcommand{\mat}[1]{\boldsymbol{\bm #1}}

\DeclareMathOperator*{\argmax}{arg\,max}

\allowdisplaybreaks

\begin{document}

\title[Continuous scaled phase-type distributions]{Continuous scaled Phase-type distributions}

\author[H. Albrecher]{Hansj\"{o}rg Albrecher}
\address{Faculty of Business and Economics,
University of Lausanne,
Quartier de Chambronne,
1015 Lausanne, and Swiss Finance Institute, Lausanne, 
Switzerland}
\email{hansjoerg.albrecher@unil.ch}

\author[M. Bladt]{Martin Bladt}
\address{Faculty of Business and Economics,
University of Lausanne,
Quartier de Chambronne,
1015 Lausanne,
Switzerland}
\email{martin.bladt@unil.ch}

\author[M. Bladt]{Mogens Bladt}
\address{Department  of Mathematics,
University of Copenhagen,
Universitetsparken 5,
DK-2100 Copenhagen,
Denmark}
\email{bladt@math.ku.dk}

\author[J. Yslas]{Jorge Yslas}
\address{Institute of Mathematical Statistics and Actuarial Science,
University of Bern,
Alpeneggstrasse 22,
CH-3012 Bern,
Switzerland}
\email{jorge.yslas@stat.unibe.ch}

\date{\today}

\begin{abstract}
Products between phase-type distributed random variables and any independent, positive and continuous random variable are studied.  Their asymptotic properties are established, and an expectation-maximization algorithm for their effective statistical inference is derived and implemented using real-world datasets. In contrast to discrete scaling studied in earlier literature, in the present continuous case closed-form formulas for various functionals of the resulting distributions are obtained, which facilitates both their analysis and implementation. The resulting mixture distributions are very often heavy-tailed and yet retain various properties of phase-type distributions, such as being dense (in weak convergence) on the set of distributions with positive support.
\end{abstract}
\maketitle

\section{Introduction}

{\color{black}A phase-type (PH) distributed random variable is defined as the time until absorption of a time-homogeneous pure-jump Markov process with a finite state-space having one absorbing state and all others transient.}
PH distributions are particularly attractive within applied probability and statistics since its members are explicitly described through functionals involving matrix exponentials, making this class both versatile and tractable, see \cite{Bladt2017} for a comprehensive account of PH distributions. 
Furthermore, the PH class is known to be dense (in the sense of weak convergence) {\color{black}among} the distributions concentrated on the positive half-line, so that one can approximate any non-negative random variable with arbitrary precision by a PH random variable.
However, PH random variables are always light-tailed (have an exponentially bounded tail), and for many applications, this is too restrictive. In particular, in various application areas, there is a focus on modeling the tail, which is often heavier than exponential. 

As an alternative, a class of discrete scaled PH distributions (NPH) was introduced in \cite{bladt2015calculation}. The NPH class consists of distributions that can be expressed as the distribution of a product $VY$, where $V$ is a discrete random variable, and $Y$ is PH distributed. Parameter estimation for the NPH class was subsequently treated in \cite{bladt2017fitting} via an expectation-maximization (EM) algorithm {\color{black}(see also \cite{asmussen1996fitting} for the unscaled case)}. Finally, in \cite{rojas2018asymptotic}, the tail behavior in the general case, when the scale component is also allowed to be continuous, was studied in detail. 
In particular, the authors showed that a scaled PH distribution is heavy-tailed if and only if the scaling distribution is unbounded.
Thus, the NPH class is of particular interest for modeling purposes due to its inherited denseness in the class of distributions on the positive half-line and its genuinely heavy tails (for a recent discussion of actuarial applications of discrete mixtures of exponential distributions, see also \cite{cossette2021univariate}). 
However, one downside of NPH distributions is that expressions such as the density and distribution functions are given in terms of infinite {\color{black}matrices}, which in practice can only be computed up to a finite number of terms.  

In this paper, we study the class of scaled PH distributions for the case when the scaling component is continuous. It turns out that under this construction, one can obtain closed-form formulas for different functionals. We provide new results on the tail behavior of these distributions, which complements the work in \cite{rojas2018asymptotic} and answers some open questions posed there. We then proceed to develop an EM algorithm for maximum-likelihood estimation of these models, which enables the use of this class of models on real-life data.  We adapt the algorithm to the case of censored data, and we show how the EM algorithm can be employed to approximate a known given theoretical distribution. A crucial difference to the NPH arises as a by-product of the closed-form formulas that we obtain for the mixed distributions: the latter can be evaluated using functional calculus tools, which avoids the truncation of the infinite series of the NPH case. Recently, \cite{furman2021} considered the scaling of PH distributions in a multivariate context with an interpretation of background risk models, in the light of which one might also see the present contribution as a theoretical underpinning of marginal properties of such models together with the development of an estimation algorithm for them when faced with real data. 

{\color{black}
However, although we are exclusively interested in products of random variables where at least one of them is PH distributed, much work has been done in more general probabilistic settings. We give a brief overview here. Breiman's lemma (cf. \cite{breiman1965some} and then extended by \cite{cline1994subexponentiality}) established the regular variation of the product of two independent random variables where one is regularly varying and the other one has a moment condition. In \cite{embrechts1980closure}, the closure property was shown, namely that the product of regularly varying variables is again regularly varying with the same tail index, although no explicit asymptotics were provided. Subsequently, \cite{cline1986convolution} (see also \cite{cline1987convolutions}) linked the asymptotics of products of independent random variables in terms of survival function ratios. More generally, \cite{cline1994subexponentiality} studied the closure property of product convolutions within the subexponential class of distributions. Under more relaxed conditions, \cite{tang2006subexponentiality} established the closure of product convolutions for a slightly smaller class than the subexponential one. The latter author also considered necessary and sufficient conditions of the product distribution to be long-tailed when one of the component variables satisfies a generalization of lattice and long-tailed distributions.
The multivariate subexponential case was studied in \cite{samorodnitsky2016multivariate}. Recently, \cite{xu2017necessary} found necessary and sufficient conditions for the subexponentiality of the product convolution, and provided a sufficient condition for the reverse problem: establishing the subexponentiality of a component given that of the product convolution. Finally, we would like to remark that the approach in the present paper is less general than the one in some of the papers above, but our focus is on a statistically tractable class that still allows for a fairly broad body and tail behavior of the distribution.
}

The rest of the paper is organized as follows. In Section~\ref{sec:ph} we present an overview of the PH class and its most important properties for our purposes. In 
Section~\ref{sec:cph} we introduce the class of continuous scaled PH distributions, provide new insights on their tail behavior (Section~\ref{sec:tail}), and derive an EM algorithm for parameter estimation (Section~\ref{sec:em}) as well as an extension to the case of censored observations (Section~\ref{sec:censored}). In Section~\ref{sec:examples} we provide some numerical illustrations with real-life data and concerning the approximation of given distributions. Finally, Section~\ref{sec:conclusion} concludes. 

\section{Preliminaries on PH distributions}\label{sec:ph}
This section presents the relevant preliminaries on PH distributions. For a random variable $X$, the notation $X\sim F$ for $F$ being a distribution function, density, or acronym, is understood as $X$ following the distribution uniquely associated with $F$. Unless stated otherwise, equalities between random objects hold almost surely. For two real-valued functions, $g,h$ the terminology $g(t)\sim h(t)$, as $t\to \infty$ means that $\lim_{t\to \infty}g(t)/h(t)=1$.

Let $ ( J_t )_{t \geq 0}$ denote a time-homogeneous Markov jump process on a state space $\{1, \dots, p,$ $p+1\}$, where states $1,\dots,p$ are transient and state $p+1$ is absorbing. Then $ ( J_t )_{t \geq 0}$ has an intensity matrix of the form
\begin{align*}
	\mat{\Lambda}= \left( \begin{array}{cc}
		\bfT &  \bft \\
		\0 & 0
	\end{array} \right),
\end{align*}
where $\bfT=(t_{kl})_{k,l=1,\dots,p}$, is a $p \times p$ sub-intensity matrix, $\bft = (t_1 ,\dots,t_p )^{\top}$ is a $p$-dimensional column vector, and $\0$ is the $p$-dimensional row vector of zeroes. Since rows of $\mat{\Lambda}$ sum to zero, we have that $\bft=- \bfT \, \bfe$, where $\bfe $ is the $p$-dimensional column vector of ones. Let $ \pi_{k} = \P(J_0 = k)$, $k = 1,\dots, p$, $\bfpi = (\pi_1 ,\dots,\pi_p )$ be the initial distribution of the chain, and assume that $\P(J_0 = p + 1) = 0$, that is, the time until absorption is necessarily positive. Then we say that the time until absorption 
\begin{align*}
	Y = \inf \{ t \geq  0 \mid J_t = p+1 \}
\end{align*}
has a phase-type distribution with representation $(\bfalp,\bfT )$ and we write $Y \sim \mbox{PH}(\bfalp,\bfT )$. 
It can be shown that the density $f_Y$ and distribution function $F_Y$ for  $Y \sim \mbox{PH}(\bfalp,\bfT )$ are given by the closed-form expressions 
\begin{gather*}
	  f_Y(y) =  \bfpi \exp({\bfT y}) \bft \,, \quad y>0 \,, \\
	 F_Y(y) = 1- \bfpi \exp({\bfT y}) \bfe \,, \quad y>0 
\end{gather*}
in terms of the matrix exponential of a matrix, defined as $$\exp(\mat{M})=\sum_{n=0}^\infty \frac{\mat{M}^n}{n!}.$$

More generally, if $g$ is any analytic function and $\mat{M}$ is a matrix, we may define  $g(\mat{M})$ by Cauchy's formula, given by
\begin{align*}
	g(\mat{M}) = \frac{1}{2 \pi i} \oint_{\Gamma} g(z)(z\mat{I} - \mat{M}) dz \,,
\end{align*}
where $\Gamma$ is a simple closed  path in $\mathbb{C}$ which encloses the eigenvalues of $\mat{M}$, see \cite[Section 3.4.]{Bladt2017} for details. {\color{black} More generally, we refer the reader to \cite{doolittle1998analytic} for a condensed treatment of functions of matrices through the Cauchy integral, and \cite{Higham2008} for other possible and equivalent methods.}

The evaluation of such expressions can be done in several ways but is not always a straightforward task, especially for high-dimensional matrices, and complex functional analysis serves mostly as a mathematical tool. When performing estimation on real-world data, the resulting matrix is often diagonalizable. Hence, alternative methods such as diagonalization (and, more generally, the Jordan decomposition) of matrices are powerful tools in this context.

The fact that the underlying Markov chain is time-homogeneous has the consequence that the sojourn times, that is, the time spent in each state at each visit, are necessarily exponentially distributed. Furthermore, the tail $\overline{F} = 1-F$ of a PH distribution is asymptotically exponential and has the following analytic expression:
\begin{align}\label{PHtail_expansion}
\overline{F}(y)=\sum_{j=1}^{m} \sum_{k=0}^{\kappa_{j}-1}y^{k} \exp({\text{Re}\left(-\lambda_{j}\right) y })\left[a_{j k} \sin \left(\text{Im}\left(-\lambda_{j}\right) y \right)+b_{j k} \cos \left(\text{Im}\left(-\lambda_{j}\right) y \right)\right] \,,
\end{align}
where $-\lambda_j$ are the eigenvalues of the Jordan blocks $\mat{J}_{j}$ of $\bfT$, with corresponding dimensions $\kappa_j$, $j = 1,\dots, m$, and $a_{jk}$ and $b_{jk}$ are constants depending on $\bfpi$ and $\bfT$. If $-\lambda$ is the largest real eigenvalue of $\bfT$ and $n$ is the dimension of the Jordan block of $\lambda$, then it is easy to see from \eqref{PHtail_expansion} that
\begin{align}\label{PHtail_asymptotic}
	\overline{F}(y) \sim c y^{n -1} \exp({-\lambda y}) \,, \quad y\to \infty \,,
\end{align}
where $c$ is a positive constant. That is, all PH distributions have exponential tails with Erlang-like second-order bias terms. Consequently, the practical modeling of heavy tails using PH distributions can be problematic when the tail behavior is of interest. Nonetheless, the above formulas serve as a building block for the analysis that will follow in the sequel. 

\section{Random scalings of PH random variables}\label{sec:cph}
In this section we consider a univariate PH random variable $Y\sim \mbox{PH}(\bfpi,\bfT)$ and define its randomly scaled counterpart
\begin{align}\label{conditional}X:= \frac{1}{\Theta}\,Y\,,\end{align}
where $\Theta$ is some positive real-valued random variable, independent of $Y$. From the probabilistic construction of a PH random variable, it is clear that such a scaling can be conditionally realized by a transformation of the time axis, so that it can be subsumed through a modified sub-intensity matrix, that is,
\begin{align}\label{conditional1}
	X \mid \Theta = \theta \sim \mbox{PH}(\bfpi, \theta \bfT)\,.
\end{align}
for every realization of $\Theta$. 
%
This simple observation allows us to now obtain the basic properties of random variables satisfying \eqref{conditional}, stated in the following result. Here, and in what follows, we denote by $$\mathcal{L}_\Theta(s) = \E[\exp(-s\Theta)],\quad s>0,$$ the Laplace transform of $\Theta$ and by $\mathcal{L}_\Theta^{\prime}(s)$ its corresponding derivative with respect to $s$. 

{\color{black} In what follows, all considered functions are analytic in the region where the eigenvalues of $-\bfT$ lie, such that we may consider evaluating them at $-\bfT$, and more generally at $-\bfT x$ for any $x>0$, without any concern. Recall also that for Laplace transforms, the analytic property in the domain of absolute convergence follows from Morera's theorem (cf \cite{rudin87}).}

\begin{proposition}\label{prop:31}
Let $X$ be given by \eqref{conditional}. Then
\begin{enumerate}[itemsep=.3cm]
\item $F_X(x)= 1 - \bfpi \mathcal{L}_{\Theta}(-\bfT x)\bfe,\quad x>0$.
\item $f_X(x)= -\bfpi  \mathcal{L}^{\prime}_{\Theta}(-\bfT x)\,\bft,\quad x>0$.
\item $\mathcal{L}_X(s)=\bfpi\E \left[\left(\dfrac{s}{\Theta}\mat{I}-\bfT\right)^{-1}\right]\,\bft,\quad s>0$.
\item $\E(X^\nu)= \E(1/\Theta^\nu)\,\Gamma(\nu + 1)\bfpi(-\bfT)^{-\nu}\bfe$, for $\nu\ge0$, provided that it is well-defined.
\end{enumerate}
\end{proposition}
\begin{proof}

(1) follows from \eqref{conditional1}, since one can obtain the tail of $X$ as follows:
\begin{align*}
	\ov{F}_{X}(x) &= \P (X >x)\\ & = \int_0^\infty \P(X>x | \Theta=\theta ) dF_\Theta(\theta)\\
	&=\bfpi \int_0^\infty  \exp({ \theta \bfT x  })  \, dF_\Theta(\theta) \,\bfe\,.
\end{align*}
Taking derivatives in the above expression yields
\begin{align*}
	f_X(x) &= - \bfpi \int_0^\infty  \theta \bfT \exp({ \theta \bfT x  })  \, dF_{\Theta}(\theta) \,\bfe\,,
\end{align*}
from which (2) easily follows.
(3) follows from
\begin{align*}
\mathcal{L}_X(s)&=\E[\exp({-s X})]\\
&=\int_0^\infty \E[\exp({-(s/\theta) Y})]dF_{\Theta}(\theta)\\
&=\int_0^\infty \bfpi ((s/\theta)\mat{I}-\bfT)^{-1}\bft\, dF_{\Theta}(\theta)\,.
\end{align*}

Finally (4) is a consequence of $\E(Y^\nu)=\Gamma(\nu + 1)\bfpi(-\bfT)^{-\nu}\bfe$, see \cite{Bladt2017}.
\end{proof}
{\color{black}
\begin{remark}\rm
If $\Theta$ has an infinitely divisible distribution $F_\Theta$, then 
\[    \mathcal{L}_{\Theta}(x) = \exp({-B(x)})  \]
for a unique Bernstein function $B$, see for instance \cite{schilling2012bernstein}, Prop.3.12. In this case, $F_X$ has representation
\[  F_X(x) = 1 - \vect{\pi} \exp({-B(-\mat{T}x)})  \vect{e} .  \]
Here $-B(-\mat{T}x)$ is a sub--intensity matrix for all $x>0$ (cf. \cite{berg1993generation}). The corresponding density is then
\[   f_X(x) = \vect{\pi} \exp({-B(-\mat{T}x)}) B^\prime (-\mat{T}x)\mat{T}\vect{e} = -  \vect{\pi} \exp({-B(-\mat{T}x)}) B^\prime (-\mat{T}x)\vect{t} . \]
The above formulas show that for a fairly broad class of mixing distributions we obtain a closed-form formula in terms of Bernstein functions for the resulting CPH distribution, which can be computationally advantageous.
\end{remark}
}

\begin{remark}\normalfont
Scaling a PH distribution with a variable whose law is of unbounded support always results in a heavy-tailed distribution (see the next subsection for a formal definition of heavy-tailedness). However, the precise nature of the tail asymptotics, including higher-order expansions, can be calculated by applying a Jordan normal form expansion in Property (1) above. Note that the resulting behavior is solely determined by the functional form of the Laplace transform of ${\Theta}$ and the eigenvalues of $\mat{T}$.
\end{remark}

Thus, we make the following formal definition of such an $X$ in terms of scaling.

\begin{definition}\normalfont
A random variable $X$ is said to have a {\em continuous scaled phase-type} (CPH)
distribution with representation $(\vect{\pi},\mat{T})$ and scaling variable $\Theta$ if its distribution function is given by
\[ F_X(x)=1 - \bfpi \mathcal{L}_{\Theta}(-\bfT x)\bfe\,,\quad x>0\,,\]
in which case $X$ is the product of $1/\Theta$ and an independent PH$(\vect{\pi},\mat{T})$ variable.
We write
$X\sim \mbox{CPH}(\vect{\pi},\mat{T},\Theta)$.
\end{definition}

{\color{black}
\begin{theorem}\normalfont 
Let $\mathcal{C}$ be a family of positive random variables in which we may find a sequence which degenerates weakly into a positive constant $c>0$. Then the class 
$$\mathcal{D}=\{Y/\Theta\mid \Theta\in\mathcal{C},Y\sim\mbox{PH}\}$$
of CPH distributions is weakly dense in the class of distributions on the positive half-line. In other words, for any given positive random variable $Z$, we may find a sequence $\{X_n\}_{n=1}^\infty\subset \mathcal{D}$ such that
$$X_n\stackrel{d}{\to}Z.$$
\end{theorem}
\begin{proof}
The proof is a simple application of convergence through the diagonal of an array. For instance, by choosing a sequence of scaling random variables $\Theta_n$ with constant mean $k$ and variances shrinking to zero, we may then use the corresponding denseness property of regular PH distributions 
to find a suitable sequence $\{Y_n\}_{n=1}^\infty$ of PH distributed random variables with $Y_n\to Z$, which in particular establishes the required convergence of the $X_n=Y_n/\Theta_n$. Alternatively, one can adapt along the same lines the elementary proof of the denseness of mixtures of Erlang distributions (which are a special case of PH distributions) to the continuously scaled case.
\end{proof}
}

We now present some examples where the distribution of $X$ has an explicit expression.  This is not only of mathematical interest, but also of practical relevance when considering their evaluation and estimation.

%
%

\begin{example}[Gamma mixing]\normalfont 
	Consider Gamma mixing with $\Theta\sim \mbox{Gamma} (\alpha , 1)$ for any positive shape parameter $\alpha$. Then
%
\begin{align}\label{gammmix}
	\ov{F}_{X}(x)
	& = \bfpi  ( \mat{I} - x \bfT )^{-\alpha} \bfe\,,\quad x>0\,,
\end{align}
which is well-defined since $x\cdot \text{Re}(\lambda)<1 $ holds for any $x>0$, and $\lambda$ any eigenvalue of $\mat{T}$. Thus $ \mat{I} - x \bfT$ is invertible. We also have that 
\begin{align*}
	f_{X}(x) 
	& = \alpha  \bfpi  ( \mat{I} - x \bfT )^{-\alpha-1} \bft\,, \quad x>0\,.
\end{align*}
 We call this distribution {\em matrix-Pareto type II} to distinguish it from the matrix-Pareto distribution introduced in \cite{albrecher2019inhomogeneous}. Note that there is a fundamental difference to the latter in that the tail behavior is specified by a scalar (shape) parameter $\alpha$, and the scale is determined by the matrix $\bfT$. In contrast, the matrix-Pareto distribution in \cite{albrecher2019inhomogeneous} has a tail behavior that depends on the eigenvalues of the matrix $\bfT$ (arising from its matrix-valued shape parameter) and a scalar scale parameter. 
 
 Observe that mixing with more general $\Theta\sim \mbox{Gamma} (\alpha , \beta)$ results in the same class, since $\beta>0$ is a scale parameter and PH distributions are closed with respect to deterministic scalings. Consequently, from this CPH class, we may obtain the classical PH class when degenerating the Gamma mixing distribution into any positive point mass (by keeping $\alpha/\beta$ constant and letting $\alpha,\beta\to \infty$), i.e., the so-called Erlangization. The latter property is of particular interest when using this class of distributions for modeling purposes.\hfill $\Box$
\end{example}

\begin{example}[Positive stable mixing]\normalfont 
	Consider L\'evy mixing with parameter $\eta >0$ and density
	$$f_{\Theta}(\theta)=\frac{\eta}{2\sqrt{\pi \theta^3}} \exp({-\eta^2/(4 \theta)}),\quad \theta>0,$$ which is a positive stable distribution with stability parameter $1/2$, then 
\begin{eqnarray*}
	\overline{F}_{X}(x)&=&\vect{\pi}\mathcal{L}_{\Theta}(-\mat{T}x)\vect{e} \\
	&=&\int_0^\infty \vect{\pi}\exp({\theta\mat{T}x})\vect{e} \frac{\eta}{2\sqrt{\pi \theta^3}}\exp({-\eta^2/(4\theta)})d \theta \\
	&=&\vect{\pi}\exp\left({\eta \left(-\sqrt{-\mat{T}}\right)\sqrt{x}}\right)\vect{e} \,.
\end{eqnarray*}
The matrix  $-\sqrt{-\mat{T}}$ is again a sub-intensity matrix, see \cite[Page 160]{Higham2008}, so $(\bfpi,-\sqrt{-\mat{T}})$ can be seen as the parameters of another PH distribution.

More generally, let us consider a positive stable random variable $\Theta$ with corresponding Laplace transform $\mathcal{L}_\Theta(s)= \exp({-s^\alpha}) $, where $\alpha \in (0,1]$. Then, 
\begin{align*}
	\overline{F}_{X}(x) = \bfpi \exp ( - (-\bfT)^\alpha x^\alpha ) \bfe \,.
\end{align*}
Again, $(\bfpi,- (-\bfT)^\alpha)$ are parameters corresponding to another PH distribution. 
	Such distributions have a Weibull-type tail behavior, and they span the same class as matrix-Weibull laws for $\alpha \in (0,1]$ as discussed in \cite[Section 4.1]{albrecher2019inhomogeneous}. Note, however, that the two parametrizations differ, {\color{black}and for a general $\alpha$}, the underlying sub-intensity structure is not maintained. 
\hfill $\Box$
\end{example}

\subsection{Tail behavior of scaled PH distributions}\label{sec:tail}
The tail behavior of the CPH class can be studied by exploiting its representation \eqref{conditional}. In this section we use as before the notation $V=1/\Theta$. Some asymptotics of this type of construction were treated in \cite{rojas2018asymptotic}. We now recall some of these results, extend them, and provide a counter-example for a conjecture posed there.

One of the main observations in \cite{rojas2018asymptotic} is that if $V$ is unbounded, then $X$ is heavy-tailed in the sense that $\limsup_{x \to \infty} \overline{F}_X(x) \exp({\epsilon x}) = \infty $ for all $\epsilon>0$. Perhaps the most well-known class of heavy-tailed distributions is the regularly varying class denoted by $\mathcal{R}$.
Recall that a random variable $Z$ and its distribution function $F_Z$ are called regularly varying with index $\alpha > 0$ (we write $F_Z\in \mathcal{R}_{-\alpha}$) if 
\begin{align*}
	\overline{F}_{Z}(x) = x^{-\alpha} L(x) \,, \quad x>0 \,,
\end{align*}
where $L$ is a slowly varying function, that is, $L(cx)/L(x) \to 1 $ as $x \to \infty$ for all $c>0$. 

 A standard result to obtain the tail behavior of a product of independent random variables, when one of the components is regularly varying, is Breiman's lemma \citep{breiman1965some}: 
\begin{lemma}
Let $F_Z\in \mathcal{R}_{-\alpha}$, and assume that $\E[Y^{\alpha + \epsilon}]<\infty$, for some $\epsilon>0$. Then $F_{ZY}\in \mathcal{R}_{-\alpha}$, and
\begin{align*}
\overline F_{ZY}(x)\sim \E[Y^{\alpha}]\overline F_{Z}(x)\,, \quad x\to \infty \,.
\end{align*}
\end{lemma}

\begin{example}\rm
For the Matrix-Pareto type II distribution with tail \eqref{gammmix} obtained by Gamma mixing, we have that $1/\Theta\in \mathcal{R}_{-\alpha}$, so that $ \bfpi  ( \mat{I} - x \bfT )^{-\alpha}\bfe\sim Cx^{-\alpha}$ as $x\to \infty$. One can, of course, also prove this directly by decomposing the matrix into Jordan blocks.\hfill $\Box$
\end{example}

\begin{example}\rm
If  $\Theta\sim \text{PH}(\bfpi_0,\bfT_0)$, we can show from straightforward calculations and using \eqref{PHtail_expansion} that $1/\Theta$ has tail behavior
\begin{align*}
\overline F_{1/\Theta}(x)\sim\sum_{j=1}^{m} \sum_{k=0}^{\kappa_{j}-1}\frac{1}{x^{k}} b^{0}_{j k}\,,\quad x\to\infty\,,
\end{align*}
where $b_{jk}^0$ are real-valued constants. This shows that $1/\Theta\in \mathcal{R}_{-\alpha}$ for some non-negative integer $\alpha$. In particular, the ratio of any two PH distributed random variables is always regularly varying of integer order.\hfill $\Box$
\end{example}

We say that $Y$ is $\alpha$-regular variation determining ($\alpha$-rvd) if $F_Z\in\mathcal{R}_{-\alpha}$ whenever $F_{YZ}\in\mathcal{R}_{-\alpha}$.  It has been noted in \cite{rojas2018asymptotic} that special cases of PH distributions are $\alpha$-rvd. A converse of Breiman's lemma was given in \cite{jacobsen2009inverse} as follows. 

\begin{proposition}
Let $Y$ be a positive random variable with $\E[Y^{\alpha + \epsilon}]<\infty$ for some $\epsilon>0$. Then $Y$ is $\alpha$-rvd if and only if
\begin{align*}
\E[Y^{\alpha + i\eta}]\neq 0\,,\quad  \eta \in \mathbb{R}\,,
\end{align*}
where $\alpha +i\eta\in\mathbb{C}$, that is, $i^2=-1$.
\end{proposition}
We can translate the above condition into the PH setting as follows.

\begin{proposition}
 Let $Y\sim\mbox{PH}(\bfpi,\bfT)$. Then $Y$ is $\alpha$-rvd
 for  $\alpha>0$ if and only if
 \begin{align}\label{rvd_ph}
 \bfpi (-\bfT^{-1})^{\alpha +i\eta}\bfe
\neq 0\,,\quad  \eta \in \mathbb{R}\,.
\end{align}
\end{proposition}
\begin{proof}
First observe that the function
\begin{align*}
w(x)= x^{\alpha +i\eta}\,,\quad x>0\,,
\end{align*}
has a well-defined Laplace transform for any $\alpha>0$, given by
\begin{align*}
\mathcal{L}_w(s)=\int_0^\infty \exp({-sx})x^{(\alpha+i\eta+1)-1}dx=\frac{\Gamma(\alpha+i\eta+1)}{s^{\alpha+i\eta+1}}\,,\quad s>0\,.
\end{align*}
Consequently, by functional calculus,
\begin{align*}
\E[w(Y)]=\bfpi \mathcal{L}_w(-\bfT)\bft=\Gamma(\alpha+i\eta+1)\bfpi (-\bfT^{-1})^{\alpha +i\eta}\bfe \,.
\end{align*}
Since the $\Gamma$-function never vanishes, the result follows.
\end{proof}

It was conjectured in \cite{rojas2018asymptotic} that any PH distribution is $\alpha$-rvd for any $\alpha>0$. Here we provide a counterexample.

	\begin{example}[Hyperexponential]\rm
	
Consider the following sub-intensity matrix corresponding to a hyperexponential PH distribution
\[   \mat{T} = \begin{pmatrix}
-1 & 0 \\
0 & -\exp({-\pi})
\end{pmatrix} . \]
Then its associated Green matrix is given by
\[  \mat{U}= -\mat{T}^{-1} = \begin{pmatrix}
1 & 0 \\
0 & \exp({\pi})
\end{pmatrix} , \]
and thus we may easily calculate its complex power as ({\color{black} cf. \cite{doolittle1998analytic}})
\[   \mat{U}^{\alpha + i} =  \begin{pmatrix}
1 & 0 \\
0 & \exp({ (\alpha +i)\pi})
\end{pmatrix}  = \begin{pmatrix}
1 & 0 \\
0 & -\exp({\alpha \pi})
\end{pmatrix}  . \]
Therefore
\[   \mat{U}^{\alpha+i}\vect{e} = \begin{pmatrix}
1 \\
-\exp({\alpha \pi})
\end{pmatrix} , \]
and {\color{black}defining} the initial vector as
\[  \vect{\pi}= \left(  \frac{\exp({\alpha \pi})}{1+\exp({\alpha \pi}) } , 1-  \frac{\exp({\alpha \pi})}{1+\exp({\alpha \pi}) }  \right)\,,  \]
we then have that
\[   \vect{\pi}\mat{U}^{\alpha+i}\vect{e} = 0 \,.  \]
In fact, we may always construct a PH distribution such that the property \eqref{rvd_ph} fails for a  pre-specified $\eta\in\mathbb{R}$. This is achieved by considering
\[   \mat{T} = \begin{pmatrix}
-1 & 0 \\
0 & -\exp({-\pi/\eta})
\end{pmatrix}    \]
and
\[   \vect{\pi} = \left(  \frac{\exp({\alpha \pi/\eta})}{1+\exp({\alpha \pi/\eta}) } , 1-  \frac{\exp({\alpha \pi/\eta})}{1+\exp({\alpha \pi/\eta}) }  \right)  , \]
and  again $\vect{\pi}\mat{U}^{\alpha+i\eta}\vect{e}=0$.
 \hfill $\Box$
\end{example}

\begin{example}\rm
{\color{black} We now provide a concrete example of an explicit distribution which multiplied by a hyperexponential is regularly varying but itself is not. 
Consider $Y$ following a hyperexponential PH distribution with parameters 
\[   \vect{\pi} = \left(  \frac{\exp({ \pi})}{1+\exp({ \pi}) } , 1-  \frac{\exp({ \pi})}{1+\exp({ \pi}) }  \right)  , \]
and 
\[   \mat{T} = \begin{pmatrix}
-1 & 0 \\
0 & -\exp({-\pi})
\end{pmatrix} . \]
Now, take $V$ with tail function given by
\begin{align*}
	\ov{F}_V (x) = \left(1 + \frac{1}{2} \sin(\log(x)) \right) x^{-1} \,.
\end{align*}
Note that this distribution is not regularly varying. Then the product $VY$ has tail distribution
\begin{align*}
	\ov{F}_{VY}(x) &= \P( VY > x) 
	= x ^{-1} \int_0^\infty y f_Y(y)dy 
	= x ^{-1} \bfpi (- \bfT)^{-1} \bfe \,,
\end{align*}
meaning that $\ov{F}_{VY}$ is regularly varying with index $-1$. 
In fact, the above construction can be seen as a special case of Example 5 in  \cite{maulik2004characterizations}. Note that this construction hinges on having at least two mixing components, in line with the previous example.
}
\end{example}





{ 
\subsubsection{Subexponentiality and other classes of heavy-tailed distributions}

We now study other types of tail behavior, such as Weibull-type and lognormal-type, which fall into the Gumbel max-domain of attraction. We consider the following two examples.
\begin{definition}[Weibull-type tails] \rm 
A distribution function $F$ is in the Weibull-type class if 
\begin{align*}
	\ov{F}(x) \sim c x^{\beta} \exp (- \lambda  x ^{\tau} )\,, \quad x\to \infty \,,
\end{align*}
	for some constants $\beta \in \mathbb{R}$ and $\tau, \, \lambda , \, c > 0$. A Weibull-type distribution is heavy-tailed if $\tau \in (0,1)$ and light-tailed otherwise.  
	If $V$ is  Weibull-type  with parameter $\tau > 0$, then $X=VY$ is Weibull-type with parameter $\tau/(\tau+1)$, see \cite{arendarczyk2011asymptotics}. In particular, and in contrast to the regularly varying case, random scaling results in a distribution with a heavier tail than each of the two components (but still of Weibull-type).
\end{definition}

\begin{definition}[Lognormal--type tails] \rm 
A distribution function $F$ is in the lognormal--type class if 
\begin{align*}
	\ov{F}(x) \sim c x^{\beta} (\log x)^{\xi}\exp (- \lambda (\log x )^{\gamma} )\,, \quad x\to \infty\,,
\end{align*}
	for some constants $\beta, \xi \in \R$, $\gamma > 1$ and $\lambda , \, c > 0$, and we write $F \in \mbox{LN}(\gamma)$. In particular, the lognormal distribution belongs to  $\mbox{LN}(2)$. In \cite{rojas2018asymptotic}, it was shown that if $V$ has a standard lognormal tail, then 
	\begin{align*}
		\frac{\P(VY>x) }{\P(V>x)} \to \infty \,, \quad x \to \infty \,,
	\end{align*}
	 and $X=VY$ is subexponential (see below). 
\end{definition}

The regularly varying class, the Weibull-type class (with $\tau \in (0,1)$) and the lognormal-type class are subclasses of the so-called subexponential class $\mathcal{S}$. Recall that $F\in\mathcal{S}$ if $\lim_{x \to \infty} \overline{F}^{*2}(x)/\overline{F}(x) = 2 $, where {\color{black}${F}^{*2}=F\ast F$} denotes the $2$-fold convolution of $F$. The property can be shown to extend to $n$-fold convolutions, $n\in\mathbb{N}$. We now focus on finding more general conditions under which a CPH distribution has subexponential tail behavior. The following result  provides sufficient conditions for subexponentiality of the product of independent random variables. 

\begin{theorem}[Theorem 2.1 in \cite{cline1994subexponentiality}]\label{closesubexp}
	Let $V$ and $Y$ be independent non-negative random variables with distribution functions $F_V$ and $F_Y$, respectively, {\color{black} and with $Y$ not degenerate at zero.} Let $F_{VY}$ be the distribution of the product $VY$. Assume that $F_V \in \mathcal{S}$. If there is a function $a: (0,\infty) \to (0, \infty)$ such that: 
	\begin{enumerate}
		\item $a(x) \uparrow \infty$ as $x \to \infty$;
		\item $x/a(x) \uparrow \infty$ as $x \to \infty$;
		\item $\lim_{x \to \infty}\ov{F}_V(x-a(x))/{\ov{F}_V(x)}=1$;
		\item $\lim_{x \to \infty}{\ov{F}_Y(a(x))}/{\ov{F}_{VY}(x)}=0$;
	\end{enumerate}
	then $F_{VY} \in \mathcal{S}$.
\end{theorem}
\begin{remark} \rm 
	A distribution function $F$ that satisfies ${\ov{F}(x \pm h(x))}\sim {\ov{F}(x)}$ as $x \to \infty$ is called $h$-insensitive.  Thus, Condition (3) in Theorem~\ref{closesubexp} translates to $F_V$ being $a$-insensitive, where $a(\cdot)$ satisfies Conditions (1) and (2). We refer to \cite{foss2011introduction} for further reading on $h$-insensitive distribution functions.
\end{remark}

The following Corollary provides sufficient conditions for subexponential behavior of scaled PH distributions. 
\begin{corollary}\label{closesubexp_ph}
	Let $V$ be a non-negative random variable with distribution function $F_V \in \mathcal{S}$, and let $Y\sim \mbox{PH}(\bfpi, \bfT )$ independent of $V$. Let $a: (0,\infty) \to (0, \infty)$ be a function such that: 
	\begin{enumerate}
		\item $a(x) \uparrow \infty$ as $x \to \infty$;
		\item $x/a(x) \uparrow \infty$ as $x \to \infty$;
		\item $\sqrt{x}/a(x)\to 0$ as $x \to \infty$;
		\item $\lim_{x \to \infty}{\ov{F}_V(x-a(x))}/{\ov{F}_V(x)}=1$.
	\end{enumerate}
	Then $X=VY$ has subexponential tail behavior.
\end{corollary}

\begin{proof}
Since $a(\cdot)$ satisfies Conditions (1)-(3) of Theorem \ref{closesubexp}, it remains to show that  Condition (4) is also satisfied. Using  $\ov{F}_{V} \left( x^{1/2}\right) \ov{F}_Y\left(  x^{1/2}\right) \leq \ov{F}_{VY} \left( x\right)$ for all $x>0$, \eqref{PHtail_asymptotic}, and  $x^{1/2}/a(x)\to 0$ as $x \to \infty$, we have that
\begin{align*}
	\lim_{x \to \infty}\frac{\ov{F}_Y(a(x))}{\ov{F}_{VY} \left( x\right)} 
	\leq \lim_{x \to \infty} \frac{1}{c x^{(n-1)/2}} \frac{\exp({-\lambda_{0} x^{1/2} })}{ \ov{F}_{V} \left( x^{1/2}\right)} \,
\end{align*}
for any $0<\lambda_{0}< \lambda$.  Since $F_V \in \mathcal{S}$, then $\exp({\epsilon x}) \ov{F}_V(x) \to \infty$ as $x \to \infty$ for all $\epsilon>0$, see \cite[Lemma 2.17]{foss2011introduction}. Thus, we conclude that 
\begin{align*}
	\lim_{x \to \infty}\frac{\ov{F}_Y(a(x))}{\ov{F}_{VY} \left( x\right)} =0 \,.
\end{align*}
It follows then, by Theorem \ref{closesubexp}, that the product $VY$ has subexponential tail behavior. 
\end{proof}

\begin{remark} \rm 
	Note that the conditions on $V$ above are solely in terms of its tail behavior, and even though they seem rather restrictive, they are satisfied by several classes of distributions, including:  intermediate regularly varying distributions, see \cite[Theorem 2.47]{foss2011introduction}, Weibull-type distributions with $\tau \in (0,0.5)$  and lognormal-type distributions with $\gamma>1$. 
	 Note also that $a(x)=M \sqrt{x}$ for some $M>1$ is possible, and then the class of distributions that satisfies the conditions of Corollary~\ref{closesubexp_ph} is closely related to the $\sqrt{x}$-insensitive class of distributions, which can be characterized by means of a convergence on probability, see \cite[Theorem 2.49]{foss2011introduction}. 
\end{remark}

{\color{black}

We now address different subclasses of $\mathcal{S}$, and we start with the $\mathcal{A}$ class of distributions. Recall that $F$ belongs to the $\mathcal{A}$ class if $F \in \mathcal{S}$ and   
\begin{align}
\label{eq:aclass}
	\limsup_{x \to \infty} \frac{\ov{F}(c x)}{\ov{F}( x)} < 1 
\end{align} 
for some $c>1$.

\begin{remark}\rm
	Condition \eqref{eq:aclass} is a mild restriction, and it is satisfied by a vast number of distributions, including those in Table 1.2.6 in \cite{embrechts2013modelling}, making this a relevant subclass of $\mathcal{S}$.
\end{remark}

Note that, in particular, $\mathcal{A} \subset \mathcal{S}$, and both are subclasses of the more general class of long-tailed distributions $\mathcal{L}$ (see the appendix for a definition). The following proposition, which characterizes the tail behavior of a scaled phase-type distribution when the scaling is long-tailed, is key for deriving more explicit tail asymptotics when dealing with different subclasses of $\mathcal{S}$.
\begin{proposition}\label{prop:long}
	Let $V$ be a non-negative random variable with distribution function $F_V \in \mathcal{L}$ and let $Y\sim \mbox{PH}(\bfpi, \bfT)$, independent of $V$. Then $\ov{F}_Y(x)=o\left( \ov{F}_{VY}(bx)\right)$ for all $b>0$.
\end{proposition}
\begin{proof}
	Let $b>0$. Since $Y\sim \mbox{PH}(\bfpi, \bfT )$, we have that
	\eqref{PHtail_asymptotic} holds.  Using \linebreak $\ov{F}_{V}\left(  \lambda_{0}x \right) \ov{F}_Y\left( {b}/{\lambda_{0}}\right) \leq \ov{F}_{VY}\left( b x \right)$ for all $x>0$, we have that 
	\begin{align*}
		\lim_{x \to \infty} \frac{\ov{F}_Y(x)}{\ov{F}_{VY}(bx)} 
	 \leq \lim_{x \to \infty}  \frac{\exp({-\lambda_{0} x})}{\ov{F}_{V}\left(  \lambda_{0}x \right) \ov{F}_Y\left( {b}/{\lambda_{0}}\right) } \,
	\end{align*}
	for any $0<\lambda_{0}<\lambda$. Given that $F_{V} \in \mathcal{L}$, then $\exp({\epsilon x}) \ov{F}_V(x) \to \infty$ as $x \to \infty$ for all $\epsilon>0$, which implies the result. 
\end{proof}

The next result provides sufficient conditions for the tail behavior in the $\mathcal{A}$ class of the product of independent random variables.
\begin{theorem}[Theorem 2.1 in \cite{tang2006subexponentiality}] \label{the:closeA}
	Let $V$ and $Y$ be independent non-negative random variables with distribution functions $F_V$ and $F_Y$, respectively, and let $F_{VY}$ be the distribution function of the product $VY$. If $F_{V} \in \mathcal{A}$ and $\ov{F}_Y(x)=o\left( \ov{F}_{VY}(bx)\right)$ for all $b>0$, then $F_{VY} \in \mathcal{A}$.
\end{theorem}

As an immediate consequence, we obtain the following result for scaled phase-type distributions. 

\begin{corollary}
	Let $V$ be a non-negative random variable with distribution function $F_V$ and let $Y\sim \mbox{PH}(\bfpi, \bfT)$, independent of $V$. If $F_V \in \mathcal{A}$, then
	${F}_{VY} \in \mathcal{A}$.
\end{corollary}
\begin{proof}
	Follows directly from Theorem~\ref{the:closeA}, Proposition~\ref{prop:long} and using that $\mathcal{A} \subset \mathcal{L}$.
\end{proof}

A necessary and sufficient condition for subexponentiality of the product was more recently derived in \cite{xu2017necessary}. Denote by $D[F]$ the set of all positive discontinuities of $F$. 

\begin{theorem}[Theorem 1.2 in \cite{xu2017necessary}]
	Let $V$ and $Y$ be independent non-negative random variables with distribution functions $F_V$ and $F_Y$, respectively, and let $F_{VY}$ be the distribution of the product $VY$. Then $F_{VY} \in \mathcal{S}$ if and only if $F_{V} \in \mathcal{S}$ and either $D[F_{V}] = \O$ or $D[F_{V}] \neq \O$ and 
	\begin{align}\label{eq:condsub}
		\ov{F}_{Y}(x / d) - \ov{F}_{Y}((x+1) / d) = o(\ov{F}_{VY}(x))\,,
	\end{align}
	for all $d \in D[F_{V}]$.
\end{theorem}
\begin{remark}\rm 
	In particular, the result above implies that when the scaling component in a scaled phase-type random variable is a continuous random variable with subexponential tail (thus, falling in the CPH class), the resulting distribution is subexponential as well. However, note that when dealing with discrete scaling, one should proceed more carefully since Condition \eqref{eq:condsub} is, in general, not easy to verify. Thus, the different criteria provided here are useful tools to determine the subexponentiality of the product, and the use of each of them depends on each particular case.   
\end{remark}

We now  consider other subclasses of $\mathcal{S}$. More specifically, we consider the classes of 
	  {\em Extended regularly varying} ($\mathcal{E}$), 
	 {\em Intermediate regularly varying} ($\mathcal{I}$),
	  {\em and Dominated varying} ($\mathcal{D}$) distributions. These classes satisfy  $\mathcal{R} \subset \mathcal{E} \subset \mathcal{I} \subset \mathcal{D}$ and  $ \mathcal{I} \subset  \left( \mathcal{D} \cap \mathcal{L} \right) \subset \mathcal{S} \subset \mathcal{L}$. 
	  We refer to the appendix for their definition.

The specific tail behavior of a PH distribution now yields the following results.
\begin{corollary}
	Let $V$ be a non-negative random variable with distribution function $F_V$ and let $Y\sim \mbox{PH}(\bfpi, \bfT)$ independent of $V$. Then we have that
	\begin{enumerate}
		\item If $F_V \in \mathcal{R}$ then ${F}_{VY} \in \mathcal{R}$ and their indices of regular variation are the same.
		\item If $F_V \in \mathcal{E}$ then ${F}_{VY} \in \mathcal{E}$.
		\item If $F_V \in \mathcal{I}$ then ${F}_{VY} \in \mathcal{I}$.
		\item If $F_V \in \mathcal{D} \cap \mathcal{L}$ then ${F}_{VY} \in \mathcal{D} \cap \mathcal{L}$.
		\item If $F_V \in \mathcal{L}$ then ${F}_{VY} \in  \mathcal{L}$.
	\end{enumerate}
\end{corollary}
\begin{proof}
	Note  that for all cases $F_{V} \in \mathcal{L}$, then Proposition~\ref{prop:long} implies that  $\ov{F}_Y(x)=o\left( \ov{F}_{VY}(bx)\right)$ for all $b>0$. Thus, 
	\begin{enumerate}
		\item Follows from \cite[Corollary 3.6 (ii)]{cline1994subexponentiality} (and is in fact also a direct consequence of Breiman's lemma).
		\item Follows from \cite[Theorem 3.5 (iii)]{cline1994subexponentiality}.
		\item Follows from \cite[Theorem 3.4 (ii)]{cline1994subexponentiality}.
		\item Follows from \cite[Theorem 3.3 (ii) and Theorem 2.2 (iii)]{cline1994subexponentiality}.
		\item Follows from \cite[Theorem 2.2 (iii)]{cline1994subexponentiality}.
	\end{enumerate}
\end{proof}
}

\subsection{An EM algorithm for CPH distributions}\label{sec:em}
We present an EM algorithm for estimating the CPH class of distributions when $\Theta$ is any positive, continuous random variable.
Assume that $\Theta$ belongs to a parametric family depending on the vector $\vect{\alpha}$ and denote by $f_\Theta(\,\cdot \,; \vect{\alpha})$ its corresponding density function. 

As is usual, we will exploit the path representation of PH distributions seen as absorption times of a finite-state Markov jump process. Thus, the complete data {\color{black}are the entire paths of the Markov jump processes} and the scaling component $\Theta$. In this setting, not only do we not observe the state sojourn times and transitions, but the realizations of $\Theta$ are also not observed.

Consider $x_1, \dots, x_M$ an iid sample from a CPH distributed random variable, and let $L_c(\bfpi, \bfT,\vect{\alpha} ;\bfx)$ denote the corresponding complete {\color{black}data} likelihood function. In order to write out the latter explicitly, we need to make the following definitions. Let $B_k$ be the number of times the underlying Markov jump {\color{black} processes start}
in state $k$, $N_{kl}$ the total number of transitions from state $k$ to $l$ until absorption, $N_k$ the number of
times that $k$ was the last state to be visited before absorption, and finally let $Z_k$ be  the  cumulated time that
the Markov jump {\color{black}processes} spent in state $k$.

Then, using $f_{X,Y}(x,y) = f_{X|Y}(x|y)f_{Y}(y)$ and conditioning on the {\color{black}paths of the Markov jump processes}
\begin{align*}
&L_c(\bfpi, \bfT, \vect{\alpha} ; \bfx ) \\ 
&\quad = f_\Theta(\theta; \vect{\alpha}) \prod_{k=1}^{p}\pi_{k}^{B_{k}}\prod_{k=1}^{p}\prod_{l=1, l\neq k}^{p} \left(\theta t_{kl}\right)^{N_{kl}}\exp\big( -\theta t_{kl}Z_{k} \big) \\
&\quad\quad\times\prod_{k=1}^{p}\left(\theta t_{k} \right)^{N_{k}}\exp\big(-\theta t_{k} Z_{k} \big)  \,.
\end{align*}
Consequently, the corresponding log-likelihood (disregarding the terms which do not depend on any parameters) is given by 
\begin{align*}
l_c(\bfpi, \bfT,\vect{\alpha} ; \bfx) 
&= \sum_{k=1}^{p} {B_{k}} \log \left( \pi_{k} \right) + \sum_{k=1}^{p}\sum_{l=1, l\neq k}^{p} {N_{kl}} \log \left(t_{kl}\right)
-\sum_{k=1}^{p}\sum_{l=1, l\neq k}^{p}{t_{kl} \theta Z_{k} } \\
& +  \sum_{k=1}^{p} {N_{k}}\log \left(t_{k}\right)
-\sum_{k=1}^{p} {t_{k} \theta Z_{k} } + \log(f_\Theta(\theta; \vect{\alpha}) )  \,.
\end{align*}

With this decomposition of the full likelihood at hand, we now outline the necessary computations that are needed in each of the two steps of the EM algorithm, and then collect the main formulas at the end.\\

\textbf{E-Step}

This step consists of computing the conditional expectation of the {\color{black}complete data} log-likelihood given the observed data, and given some fixed parameters. 
We consider one (generic) data point ($M = 1$) and let $x = x_1$. Then
\begin{align*}
	\E \left[ \log(f_\Theta(\Theta; \vect{\alpha}) ) \mid X=x \right] & = \int_0^\infty \log(f_\Theta(\theta; \vect{\alpha}) ) f_{\Theta | X} (\theta | x) d\theta \\
	& =   \int_0^\infty \log(f_\Theta(\theta; \vect{\alpha}) ) \frac{f_{\Theta , X} (\theta , x)}{f_{X}(x)} d\theta \\
	& =   \int_0^\infty \log(f_\Theta(\theta; \vect{\alpha}) ) \frac{f_{ X | \Theta } ( x | \theta ) f_\Theta(\theta)}{f_{X}(x)} d\theta \\
	& = \frac{\int_0^\infty \log(f_\Theta(\theta; \vect{\alpha}) ) \, \bfpi \,\exp({ \theta \bfT x}) \theta\, \bft \, f_\Theta(\theta) d\theta}{f_{X}(x)} \,.
\end{align*}
Regarding $B_{k} $, 
\begin{align*}
	\E \left[ B_{k} \mid X =x \right] & =  \P \left(  J_{0}=k \mid X =x\right)  \\
	& = \int_0^\infty \P \left(J_{0}=k, \Theta \in d \theta  \mid X =x\right)  d\theta \\
	& = \int_0^\infty \frac{\P \left(J_{0}=k, \Theta \in d \theta , X \in dx\right)}{\P \left(X \in dx\right)}   d\theta \\
	& =  \frac{ \int_0^\infty \P \left( X \in dx \mid J_{0}=k, \Theta =\theta \right) \P \left(J_{0}=k \mid  \Theta= \theta \right) f_\Theta(\theta) d\theta}{\P \left(X \in dx\right)}  \\
	& =  \frac{ \int_0^\infty \pi_k \, \bfe^{\top}_{k} \exp({\theta \bfT x }) \theta\, \bft \, f_\Theta(\theta) d\theta}{f_{X}(x)}   \,,
\end{align*}
where $\bfe_k$ denotes a $p$-dimensional column vector with all entries equal
to zero except the $k$-th entry, which equals one.

For the term which involves the product of $\Theta$ and $Z_k$ we will make use of the tower property of conditional expectations, namely  that
\begin{align*}
	\E \left[ \Theta Z_{k} \mid X =x \right] = \E \left( \Theta \E \left(  Z_{k} \mid \Theta, X = x \right) \mid X =x\right)\,.
\end{align*}
Since
\begin{align*}
	\E \left[  Z_{k} \mid \Theta =\theta, X=x \right] & =  \E \left.\left[ \int_0^\infty 1_{\lbrace J_{u}=k \rbrace}du \right| \Theta = \theta, X = x\right] \\
	&  = \int_0^\infty  \P  \left( J_{u}=k  \mid \Theta = \theta, X =x\right) du \\
	& = \frac{\int_0^\infty  \P \left( J_{u}=k  , \Theta \in d \theta, X \in dx\right) du}{  \P \left(\Theta \in d \theta, X \in dx\right) }\,,
\end{align*}
then 
\begin{align*}
	\E &\left[ \Theta Z_{k} \mid X=x \right]\\
	&  = \int_0^\infty \theta  \frac{\int  \P \left( J_{u}=k  , \Theta \in d \theta, X \in dx\right) du}{  \P \left( \Theta \in d \theta, X \in dx\right) } \P \left( \Theta \in d \theta \mid  X =x\right) d\theta \\
	& = \int_0^\infty \theta  \frac{\int  \P \left( J_{u}=k  , \Theta \in d \theta, X \in dx\right) du}{  \P \left( X \in dx\right) } d\theta \\
	& = \frac{\int_0^\infty \theta \int_0^\infty   \P \left(  X \in dx \mid J_{u}=k  , \Theta=\theta\right) \P \left( J_{u}=k  \mid  \Theta =\theta\right) f_\Theta(\theta)  du d\theta}{\P \left( X \in dx\right)} \\
	& = \frac{\int_0^\infty \theta \int^{x}_{0}  \bfe^{\top}_{k} \exp({\theta \bfT(x-u)}) \theta \bft \bfpi \exp({\theta \bfT u })\bfe_{k}   du  f_\Theta(\theta) d\theta}{f_{X}(x)}\,.
\end{align*}
Similarly, one finds that
\begin{align*}
	\E &\left[ N_{kl} \mid X =x \right] = \frac{\int \theta t_{kl} \int^{x}_{0}  \bfe^{\top}_{l} \exp({\theta \bfT(x-u)})\theta \bft \bfpi \exp({\theta \bfT u })\bfe_{k}   du  f_\Theta(\theta) d\theta}{f_{X}(x)}\,,
\end{align*}
and 
\begin{align*}
	\E \left[ N_{k} \mid X =x \right]
	& = \frac{\int_0^\infty \theta t_{k}  \bfpi \exp({\theta \bfT x })\bfe_{k}    f_\Theta(\theta) d\theta}{f_{X}(x)}\,.
\end{align*}
For $M>1$, we simply sum over $x_i$, $i=1,\dots,M$, in the formulas above. 

\textbf{M-Step}

Having found the required expectations, the maximization of the conditional expected log-likelihood in terms of the parameters $\vect{\alpha}$, $\bfpi$ and $\bfT$ is done separately. Some of the expressions will in general not have an explicit solution.

In full generality, for the parameter $\vect{\alpha}$ of the scaling distribution we write
\begin{align*}
	\hat{\boldsymbol{\alpha}}= \argmax_{\boldsymbol{\alpha}} \E \left[ \log(f_\Theta(\Theta; \vect{\alpha}) ) \mid \bfX = \bfx \right] \,.
\end{align*}
For the parameters of the PH component, we first address the estimation of $\bfpi$. Consider the Lagrange function
\begin{align*}
	\zeta(\bfpi)=\sum_{k=1}^{p} {B_{k}} \log \left( \pi_{k} \right)+ \mu \left(1-\sum_{k=1}^{p} \pi_{k} \right),
\end{align*}
where $\mu$ is a Lagrange multiplier. Then
\begin{align*}
	\frac{\partial \zeta(\bfpi)}{\partial \pi_{k}}= \frac{{B_{k}}}{\pi_{k}} - \mu =0 \iff \mu \pi_{k}= B_{k}\,.
\end{align*}
Summing over $k$
\begin{align*}
	\mu=  \sum_{k=1}^{p}B_{k}=M\,,
\end{align*}
thus
\begin{align*}
 \hat{\pi}_{k}= \frac{{B_{k}}}{M}\,.
\end{align*}

Now, we consider the non-diagonal elements of $\bfT$, $t_{kl}$, $k \neq l$. We have that
\begin{align*}
	\frac{\partial l_{c}}{\partial t_{kl}}= N_{kl} \frac{1}{t_{kl}}- \theta Z_{k}=0\,,
\end{align*}
implying 
\begin{align*}
	\hat{t}_{kl}=\frac{N_{kl} }{\theta Z_{k}}\,.
\end{align*}
In a similar way, we can show that for $t_{k}$
\begin{align*}
	\hat{t}_{k}=\frac{N_{k} }{\theta Z_{k}}\,.
\end{align*}
Finally, we obtain the estimators for the diagonal elements of $\bfT$, by
\begin{align*}
	\hat{t}_{kk}=-\left( \hat{t}_{k} + \sum_{l\neq k} \hat{t}_{kl}\right).
\end{align*}
We summarize the two steps into a completed EM algorithm.

\begin{algorithm}[EM algorithm for CPH distributions]\label{alg:EMCPH}
\

	0. Initialize with some ``arbitrary'' $( \bfpi,\bfT, \vect{\alpha})$.
	
	1. (E-step) Calculate

\begin{align*}
	& \E \left( B_{k} \mid \bfX = \bfx \right) 
	=\sum_{n=1}^{M} \int_0^\infty \frac{  \pi_k \, \bfe^{\top}_{k} \exp({\theta \bfT x_n }) \theta \bft }{f_{X}(x_n)} \, f_\Theta(\theta) d\theta,   \\[3mm]
	&\E \left( \Theta Z_{k} \mid \bfX = \bfx \right)
	 = \sum_{n=1}^{M} \int_0^\infty  \theta \frac{ \int^{x_n}_{0}  \bfe^{\top}_{k} \exp({\theta \bfT(x_n-u)})\theta \bft \bfpi \exp({\theta \bfT u })\bfe_{k}   du }{f_{X}(x_n)}  f_\Theta(\theta) d\theta,  \\[3mm]
	& \E \left( N_{kl} \mid \bfX = \bfx \right)
	 = \sum_{n=1}^{M} \int_0^\infty \theta t_{kl} \frac{ \int^{x_n}_{0}  \bfe^{\top}_{l} \exp({\theta \bfT(x_n-u)})\theta \bft \bfpi \exp({\theta \bfT u })\bfe_{k}   du }{f_{X}(x_n)}  f_\Theta(\theta) d\theta, \\[3mm]
	& \E \left( N_{k} \mid \bfX = \bfx \right)
	 = \sum_{n=1}^{M} \int_0^\infty \theta t_{k}   \frac{\bfpi \exp({\theta \bfT x_n })\bfe_{k}    }{f_{X}(x_n)} f_\Theta(\theta) d\theta.
\end{align*}

	2. (M-step) Let
	\begin{align*}
	\hat{\vect{\alpha}} &= \argmax_{\vect{\alpha}} \E \left( \log(f_\Theta(\Theta; \vect{\alpha}) ) \mid \bfX = \bfx \right) \\
	 &= \argmax_{\vect{\alpha}} \sum_{n=1}^{M} \int_0^\infty \log(f_\Theta(\theta; \vect{\alpha}) ) \,  \frac{ \bfpi \exp({\theta \bfT x_n}) \theta \bft \, }{f_{X}(x_n)}f_\Theta(\theta) d\theta,
	\end{align*}
	\begin{align*}
		&\hat{\pi}_{k} = \frac{\E\left( B_{k} \mid \bfX = \bfx \right)}{M}  
		\,, \quad
		\hat{t}_{kl} = \frac {\mathlarger{ \E\left( N_{kl} \mid \bfX = \bfx \right) }}{\mathlarger{ \E \left(\Theta Z_{k} \mid \bfX = \bfx \right) }}
		\,, \quad
		\hat{t}_{k} = \frac {\mathlarger{ \E\left( N_{k} \mid \bfX = \bfx \right) }} {\mathlarger{ \E\left( \Theta Z_{k} \mid \bfX = \bfx \right)}}
		\,, \\
		&\hat{t}_{kk} = -\sum_{l \neq k} \hat{t}_{kl} -\hat{t}_{k} \,.
	\end{align*}
	Let $\hat{\bfpi} = ( \hat{\pi}_{1}, \ldots , \hat{\pi}_{p} )$, $\hat{\bfT} = \{ \hat{t}_{kl} \}_{ k, l = 1, \ldots, p}$, and $\hat{\bft} = ( \hat{t}_{1}, \ldots, \hat{t}_{p} )^{ \top }$.
	
	3. Assign $\vect{\alpha} = \hat{\vect{\alpha}} $, $\bfpi:=\hat{\bfpi}$, $\bfT :=\hat{\bfT}$, $ \bft :=\hat{\bft}$ and GOTO 1.
\end{algorithm}
\begin{remark}\rm
	In general, none of the integrals will have explicit solutions, and approximations will have to be employed. Namely, one may discretize the continuous distribution and approximate the integrals by numerical methods, such as Simpson's rule, or use diagonalization for the integration of matrix functions. In practice, a combination of the two approaches yields good results.
\end{remark}

\begin{remark}\rm
	The discrete case follows in a similar way, where sums replace integrals and are much easier to handle. Such an EM algorithm for NPH distributions was introduced in \cite{bladt2017fitting}, and the above formulas can be seen as the limit when the discretization becomes infinitely fine. In \cite{bladt2017fitting}, in fact, several of the illustrations arise as discretized continuous random variables. 
\end{remark}

\subsection{Censored data}\label{sec:censored}
We call a data point \textit{right-censored} at $v$ if it takes an unknown value above $v$, and \textit{left-censored} at $w$ if it takes an unknown value below $w$. It is called  \textit{interval-censored} if it takes an unknown value within the interval $(v , w]$. Left-censoring is a particular case of interval-censoring with $v=0$, while right-censoring is obtained by fixing $v$ and letting $w \to \infty$.

The EM algorithm for censored data works in much the same way as for uncensored data, and the only change in Algorithm~\ref{alg:EMCPH} is in the E-step. The derivation follows the approach taken for the PH case, see \cite{olsson1996estimation}. We consider first
\begin{align*}
	\E \left[ \log(f_\Theta(\Theta; \vect{\alpha}) ) \1(X >v)  \right] 
	& = \int_{v}^{\infty} \E \left[ \log(f_\Theta(\Theta; \vect{\alpha}) ) \mid X = x  \right]  f_{X}(x) dx  \\
	& =  \int_{v}^{\infty}  \int_0^\infty \log(f_\Theta(\theta; \vect{\alpha}) ) f_{\Theta | X} (\theta | x) d\theta f_{X}(x) dx  \\ 
	& = \int_{v}^{\infty}  \int_0^\infty \log(f_\Theta(\theta; \vect{\alpha}) )  \bfpi \exp({ \theta \bfT x}) \theta \bft  f_\Theta(\theta) d\theta dx \\
	& =  \int_0^\infty \log(f_\Theta(\theta; \vect{\alpha}) )  \bfpi \exp({ \theta \bfT v}) \bfe  f_\Theta(\theta) d\theta  \,.
\end{align*}
Then, for interval-censored data, we obtain
\begin{align*}
	\E \left[ \log(f_\Theta(\Theta; \vect{\alpha}) ) \mid  X \in (v,w]  \right]  
	& = \frac{\E \left[ \log(f_\Theta(\Theta; \vect{\alpha}) ) \1(X >v)  \right]  - \E \left[ \log(f_\Theta(\Theta; \vect{\alpha}) ) \1(X >w)  \right] }{\P (X \in (v, w] )}\\
	& = \frac{\int_0^\infty  \log(f_\Theta(\theta; \vect{\alpha}) )  (\bfpi \exp({ \theta \bfT v}) \bfe  - \bfpi \exp({ \theta \bfT w}) \bfe  ) f_\Theta(\theta) d\theta}{\P (X \in (v, w] ) } \,.
\end{align*}

In the following we present the resulting remaining formulas (their derivations are similar to the above and are thus omitted for brevity). 
\begin{align*}
	& \E \left( B_{k} \mid  X \in (v,w] \right) 
	= \int_0^\infty \frac{  \pi_k \, \bfe^{\top}_{k} \exp({\theta \bfT v }) \bfe  - \pi_k \, \bfe^{\top}_{k} \exp({\theta \bfT w }) \bfe }{\P (X \in (v, w] )} \, f_\Theta(\theta) d\theta,   \\[3mm]
	&\E \left( \Theta Z_{k} \mid  X \in (v,w]  \right) \\
	& \quad = \int_0^\infty  \frac{\theta}{\P (X \in (v, w] )} \Bigg[  \int_{v}^{w} \bfpi \exp({\theta\bfT u}) \bfe_{k}du -  \int_{0}^{w} \bfe_{k}^{\top} \exp({\theta \bfT(w-u)}) \bfe \bfpi \exp({\theta\bfT u}) \bfe_{k}du \\
	& \quad \qquad  \qquad \qquad \qquad \qquad + \int_{0}^{v} \bfe_{k}^{\top} \exp({\theta\bfT(v-u)}) \bfe \bfpi \exp({\theta\bfT u}) \bfe_{k}du \Bigg] f_\Theta(\theta) d\theta,  \\[3mm]
	 & \E \left( N_{kl} \mid  X \in (v,w]  \right) \\ 
	 &  \quad =\int_0^\infty \frac{ \theta t_{kl}}{\P (X \in (v, w] )} \Bigg[ \int_{v}^{w}  \bfpi \exp({\theta\bfT u}) \bfe_{k}du - \int_{0}^{w} \bfe_{l}^{\top} \exp({\theta\bfT(w- u)}) \bfe \bfpi \exp({\theta\bfT u}) \bfe_{k} du \\
	 & \quad \qquad  \qquad \qquad \qquad \qquad  + \int_{0}^{v} \bfe_{l}^{\top} \exp({\theta\bfT(v- u)}) \bfe \bfpi \exp({\theta\bfT u}) \bfe_{k} du \Bigg] f_\Theta(\theta) d\theta, \\[3mm]
	& \E \left( N_{k} \mid  X \in (v,w]  \right)
	 =  \int_0^\infty \theta t_{k}   \frac{ \int_{v}^{w}  \bfpi \exp({\theta \bfT u })\bfe_{k} du    }{\P (X \in (v, w] )} f_\Theta(\theta) d\theta \,.
\end{align*}

\begin{remark}\normalfont
The main challenge when implementing the above formulas is obtaining numerical estimates for integrals defined with respect to matrix exponentials, which are usually slow to evaluate. A suite of fast and useful routines for calculating matrix exponentials (and more generally, estimation tools for PH distributions) can be found, for instance, in the R package \texttt{matrixdist}, see \cite{bladt2021matrixdist,matrixdist}.

\end{remark}

\section{Examples}\label{sec:examples}

In this section we present three detailed numerical illustrations of the estimation of CPH distributions via the EM algorithm from the previous section.
In the first two examples we fit CPH distributions to real data sets, the second one containing censored observations, while in the last example we consider the estimation of a CPH distribution to a theoretical given distribution.
In all cases, we ran the algorithms until the changes in the successive log-likelihoods became negligible. Note that the purpose of this section is not to thoroughly compare to and outperform other models for the given data (which with criteria like AIC or BIC would in any case not be evident given the non-identifiability {\color{black}and overparametrization issues} of PH distributions in general). Instead, our aim is to illustrate how the algorithms developed in this paper can, in fact, be implemented in a straightforward manner. As a consequence, we showcase CPH distributions as interesting  alternatives in the statistical modeling toolkit of the respective application areas. Note that the concrete purpose may then decide which model one might want to use.

 \subsection{Dutch fire insurance data} 
We consider claims above 1 million (Euro) from the Dutch fire insurance claim data set studied in \cite{albrecher2017reinsurance}, subtract 1 million to all data points (to shift them to the origin), and scale by a factor of $10^{-6}$. Subsequently, we fit a matrix-Pareto type II to the resulting sample. To reduce the number of parameters, we consider a Coxian structure (which is often sufficient) of dimension $3$ in the PH component, obtaining in this way the following estimated parameters
\begin{gather*} 
	\hat{\bfpi}=\left(
	1, \,0,\, 0\right)\,, \\ 
	\hat{\bfT}=\left( \begin{array}{ccc}
	-0.8620  & 0.8079  & 0  \\
	0 &  -2.4341 & 1.1014  \\
	0 & 0 & -1.5808 \\
	\end{array} \right) \,, \\ 
	\hat{\alpha}= 1.3792 \,,
\end{gather*}
and corresponding log-likehood of $-2,657.666$. 
For reference, we also consider a matrix-Pareto (type I) model \citep{albrecher2019matrix} defined by
\begin{align*}
	\ov{F}(x) = \bfpi \left( \frac{x}{\beta} + 1\right)^{\bfT} \bfe \,, \quad x>0 \,.
\end{align*}
We again use a Coxian structure of dimension 3 in the PH component. The resulting estimated parameters are
\begin{gather*} 
	\hat{\bfpi}=\left(
	1, \,0,\, 0\right)\,, \\ 
	\hat{\bfT}=\left( \begin{array}{ccc}
	-1.4370 & 1.3183   & 0  \\
	0 &  -9.0824 & 2.3177  \\
	0 & 0 & -5.8812 \\
	\end{array} \right) \,, \\ 
	\hat{\beta} = 1.5821 \,,
\end{gather*}
with corresponding log-likelihood $-2,657.702$. 

Figures ~\ref{fig:dutch} and \ref{fig:dutchqq} show that both fitted distributions provide adequate and very similar models for the sample. Regarding the tail behavior,
 we have that the index of regular variation of the matrix-Pareto type II model is given by $\hat{\alpha} = 1.3792$, while for the matrix-Pareto (type I) model, it is given by the negative of the largest real eigenvalue, viz. $\lambda = 1.4370$. 
 Note that both these two estimates are comparable with results obtained in previous studies. For instance,
in \cite[Page 107]{albrecher2017reinsurance}, a value of $1.255$ is proposed in their splicing model. 
In contrast to traditional extreme value techniques, the present matrix models are global. Hence, the tail index is only one aspect of the distribution, not the calibration focus, yet the fit is very satisfactory. As previously noted, when comparing matrix-Pareto type I and type II distributions, the latter has a tail parameter that does not depend on the underlying Markov structure. The advantage, in this case, is that it may also estimate the tail index in a first step separately with extreme value techniques and then estimate the other parameters in a subsequent step (similar to the approach in \cite{bladt2017fitting} for the NPH case). On the other hand, the matrix-Pareto type II is computationally more demanding. The running times and resulting estimators depend on the quality of the approximations in the EM algorithm formulas. For the Dutch fire insurance data, in both cases, increasing the dimension of the PH component does not improve the fit significantly. For instance, for a Coxian structure of dimension $5$ in the PH component, we obtain a log-likelihood of $-2,657.603$ in the matrix-Pareto type I model and $-2,657.653$ in the type II. Thus, the choice of dimension 3 seems eventually very reasonable here. 

\begin{figure}[h]
\centering
\includegraphics[width=0.49\textwidth]{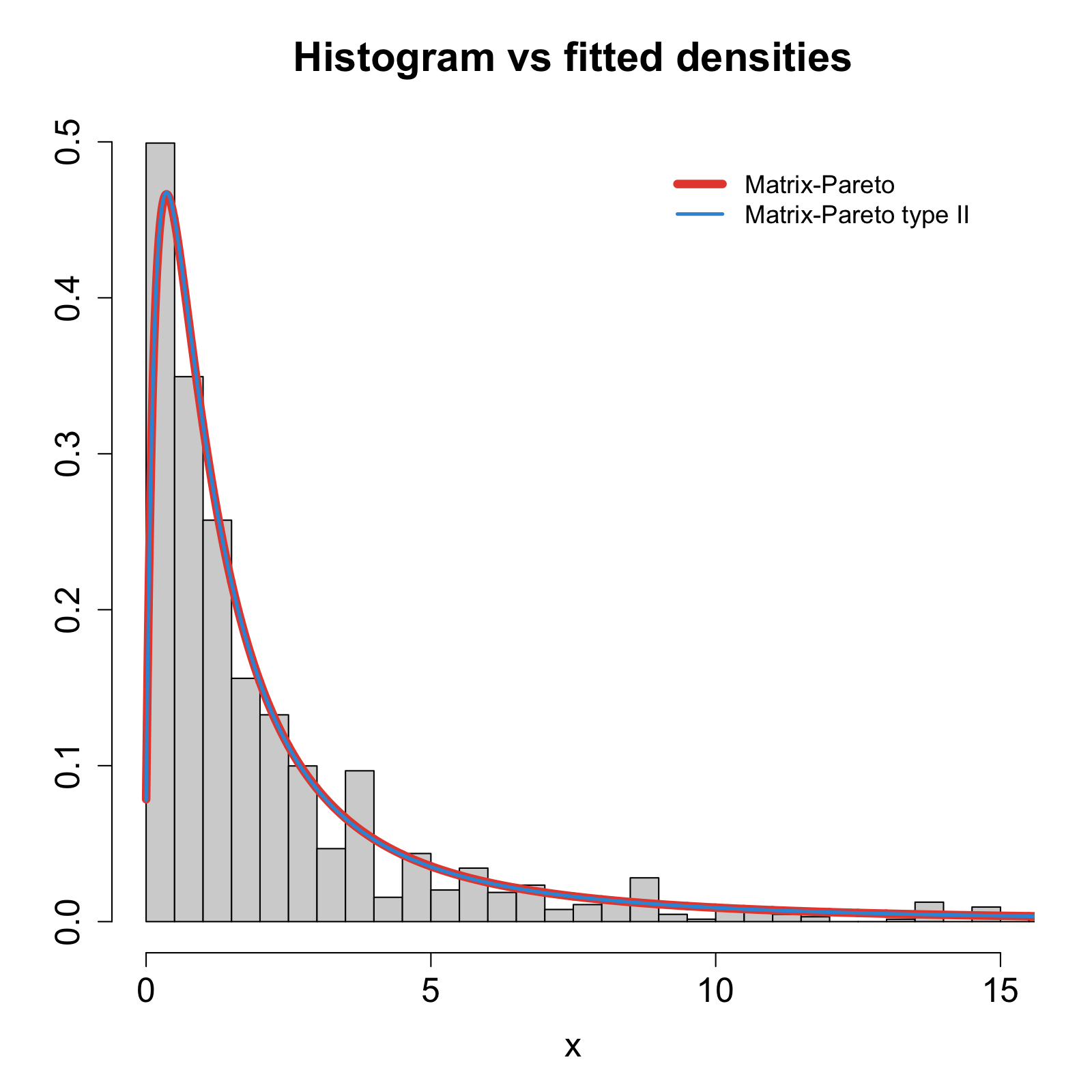}
\includegraphics[width=0.49\textwidth]{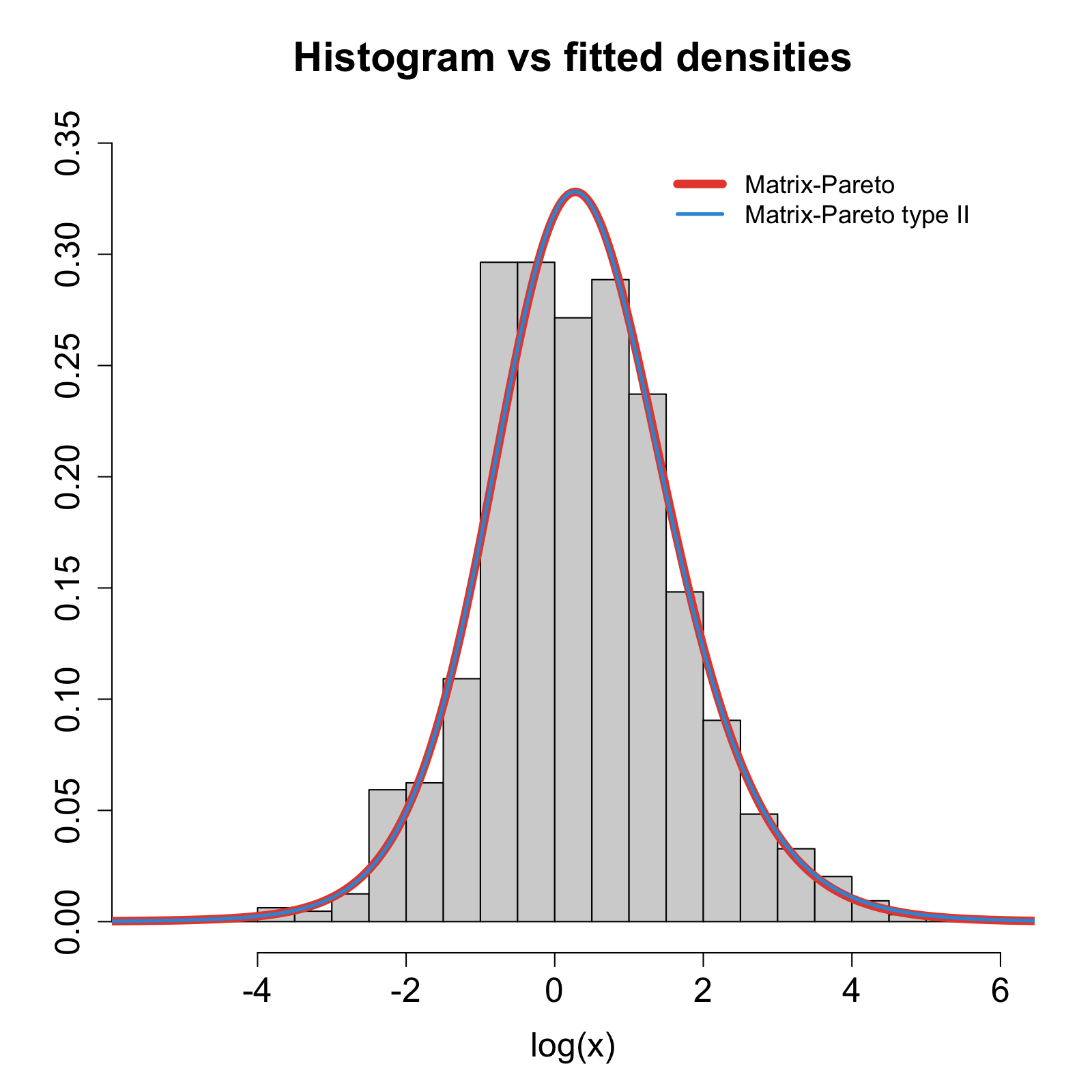}
\caption{Histogram of Dutch fire insurance data versus fitted matrix-Pareto type I and II distributions (left), and histogram of the log-data (right)}
\label{fig:dutch}
\end{figure}
\begin{figure}[h]
\centering
\includegraphics[width=0.49\textwidth]{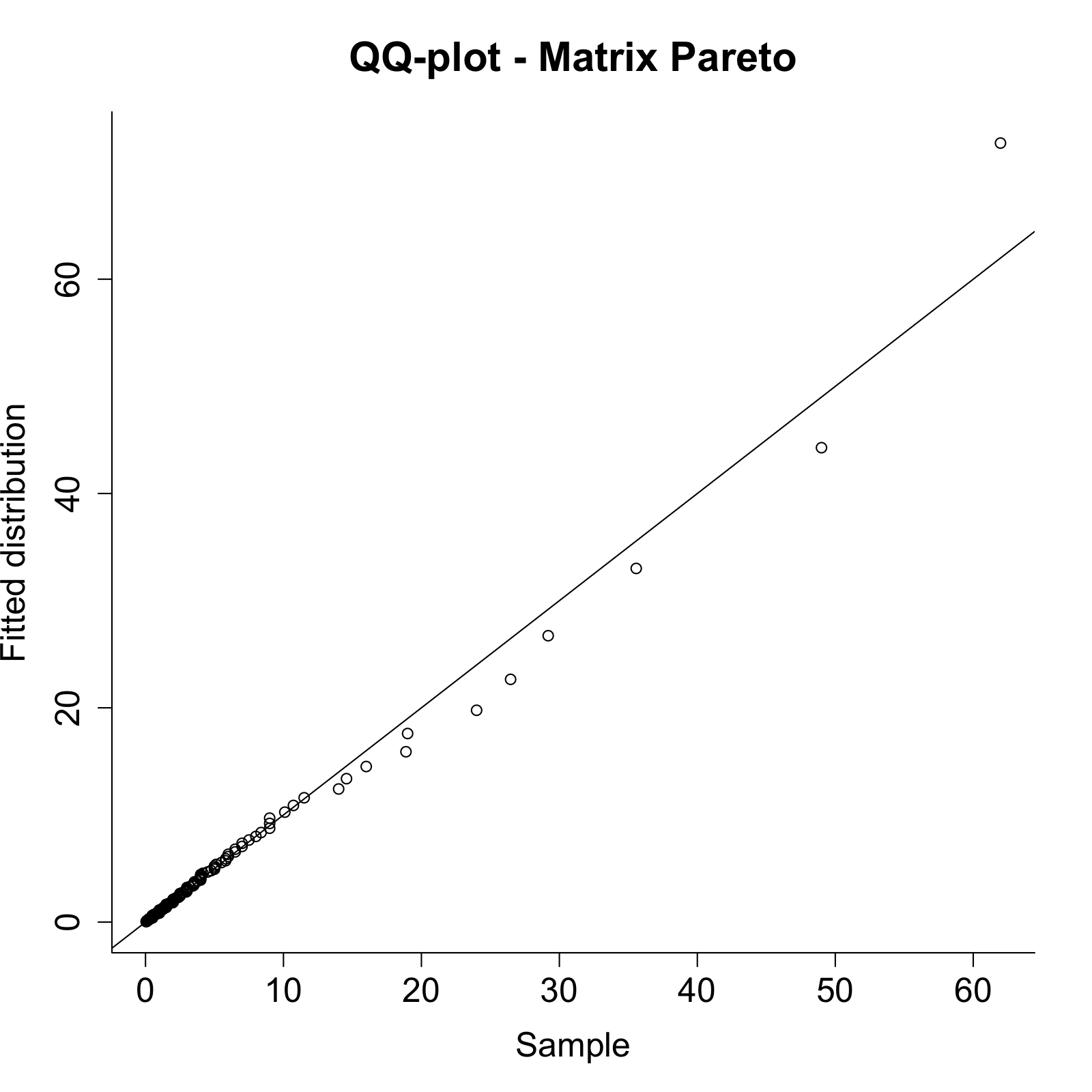}
\includegraphics[width=0.49\textwidth]{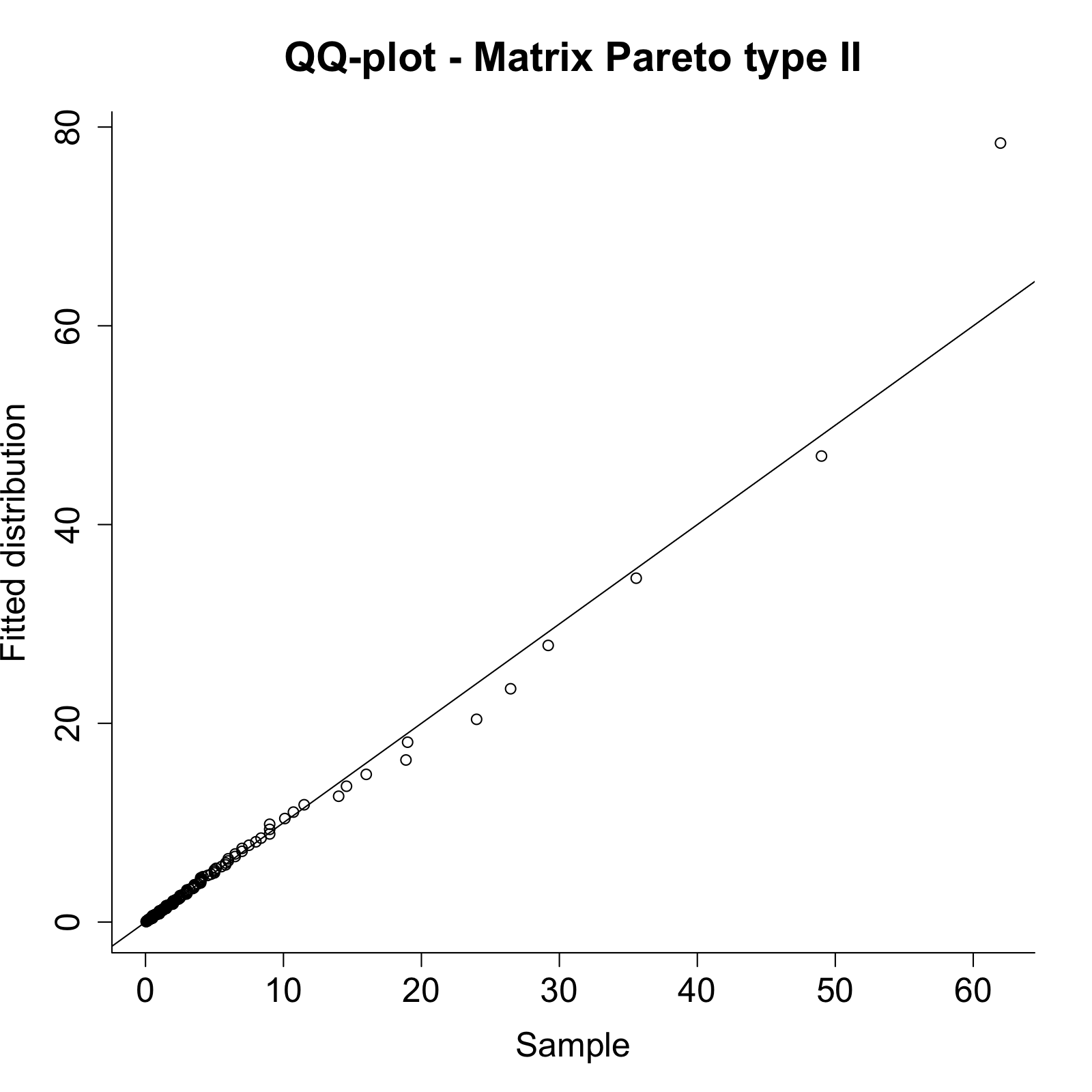}
\caption{QQ-plot of Dutch fire insurance data versus fitted matrix-Pareto type I distribution (left), and QQ-plot of Dutch fire insurance data versus fitted matrix-Pareto type II distribution  (right)}
\label{fig:dutchqq}
\end{figure}
Note that in \cite{albrecher2019inhomogeneous}, a previous analysis of the same data set employing the matrix-Pareto (type I) distribution was considered with a specific and sparse matrix structure of dimension 20, whereas here, we are interested in the best fit of a low-dimensional structure.

\subsection{Loss insurance data} 
We consider the loss insurance claim data set from \cite{frees1998understanding}.
The data comprises 1500 insurance claims from a real-life insurance company, where each data point consists of an indemnity payment (loss) and an allocated loss adjustment expense (ALAE). 
For the present analysis, we only use the loss component (scaled by a factor of $10^{-4}$), for which 34 observations are right-censored, and fit a matrix-Pareto type II of $4$ phases.  The resulting fitted parameters are 
\begin{gather*} 
	\hat{\bfpi}=\left(
	0.0476, \,0.0289,\, 0.1412, \, 0.7823\right)\,, \\ 
	\hat{\bfT}=\left( \begin{array}{cccc}
	-2.9587 & 0.1886 & 1.2395 & 0.6833  \\
	0.5585 & -3.5859 & 0.6233 & 0.0364  \\
	0.1152 & 0.0650 & -0.5554 & 0.2892 \\
	0.5079 & 1.9315 & 0.4666 & -3.0784
	\end{array} \right) \,, \\ 
	\hat{\alpha}= 1.3744 \,.
\end{gather*}
The quality of the fit is supported by Figure~\ref{fig:loss}, where we see that the cumulative hazard of the matrix-Pareto type II model is close to the one implied by the non-parametric Kaplan-Meier estimator.  A recent analysis of the same data set employing the matrix-Pareto (type I) distribution can be found in \cite{bladt2021matrixdist}.
\begin{figure}[h]
\centering
\includegraphics[width=0.49\textwidth]{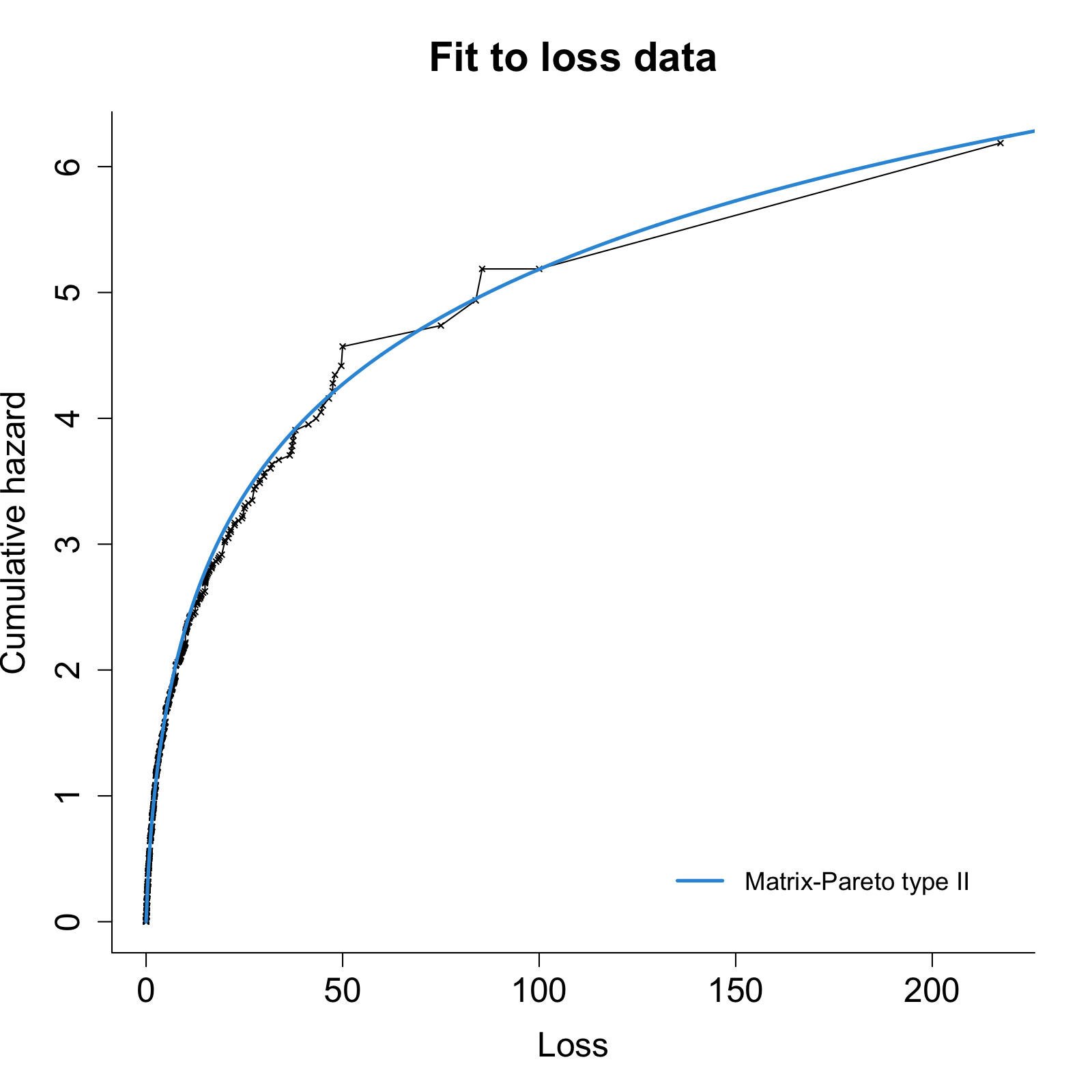}
\caption{Cumulative hazard function of the fitted matrix-Pareto type II versus the non-parametric Nelson-Aalen estimator of the sample. }
\label{fig:loss}
\end{figure}

\subsection{Fitting to a known distribution}
Algorithm~\ref{alg:EMCPH} can be modified to fit a CPH distribution to a theoretical given distribution. This is done in a similar way as in the PH case, and the idea consists of considering sequences of empirical distributions with increasing sample size. Let $h$ be an absolutely continuous density. Then, Algorithm~\ref{alg:EMCPH} and dominated convergence yield: 
\begin{align*}
	& \hat{\vect{\alpha}} \to  \argmax_{\vect{\alpha}} \int_0^\infty \int_0^\infty \log(f_\Theta(\theta; \vect{\alpha}) ) \,  \frac{ \bfpi \exp({\theta \bfT x}) \theta \bft \, }{f_{X}(x)}f_\Theta(\theta) d\theta h(x)dx \,, \\[3mm]
	&\hat{\pi}_{k} \to \int_0^\infty  \int_0^\infty \frac{  \pi_k \, \bfe^{\top}_{k} \exp({\theta \bfT x }) \theta \bft }{f_{X}(x)} \, f_\Theta(\theta)   d\theta h(x) dx
		\,, \\[3mm]
	& \hat{t}_{kl} \to  \frac {\mathlarger{\int_0^\infty  \int_0^\infty \theta t_{kl} \frac{ \int^{x}_{0}  \bfe^{\top}_{l} \exp({\theta \bfT(x-u)})\theta \bft \bfpi \exp({\theta \bfT u })\bfe_{k}   du }{f_{X}(x)}  f_\Theta(\theta) d\theta  h(x) dx }}{\mathlarger{ \int_0^\infty \int_0^\infty   \theta \frac{ \int^{x}_{0}  \bfe^{\top}_{k} \exp({\theta \bfT(x-u)})\theta \bft \bfpi \exp({\theta \bfT u })\bfe_{k}   du }{f_{X}(x)}  f_\Theta(\theta) d\theta h(x) dx }}
		\,, \\[3mm]
	& \hat{t}_{k} \to  \frac {\mathlarger{ \int_0^\infty  \int_0^\infty \theta t_{k}   \frac{\bfpi \exp({\theta \bfT x })\bfe_{k}    }{f_{X}(x)} f_\Theta(\theta) d\theta h(x) dx }} {\mathlarger{ \int_0^\infty \int_0^\infty \theta \frac{  \int^{x}_{0}  \bfe^{\top}_{k} \exp({\theta \bfT(x-u)})\theta \bft \bfpi \exp({\theta \bfT u })\bfe_{k}   du }{f_{X}(x)}  f_\Theta(\theta) d\theta h(x) dx }}
		\,.
\end{align*}

As an example, we consider a matrix-Weibull distribution \citep{albrecher2019inhomogeneous}, whose density function is given by
\begin{align*}
	h(x) = \bfpi \exp({\mat{S} x^\beta}) \vect{s} \beta x^{\beta -1}\,, \quad x > 0, \,.
\end{align*}
where $\mat{S}$ is a sub-intensity matrix, and $\vect{s}=-\mat{S}\bfe$.
For the present illustration, we take parameters
\begin{gather*} 
	{\bfpi}=\left(
	0.5, \,0.3,\, 0.2\right)\,, \\ 
	{\mat{S}}=\left( \begin{array}{ccc}
	-1 &  1 & 0  \\
	0 & -2 & 1  \\
	0 & 0 & -5 \\
	\end{array} \right) \,, \\ 
	\beta = 0.75  \,.
\end{gather*}
Then, we approximate this distribution with a CPH distribution of 3 phases with a general Coxian structure in the PH component and positive stable mixing. The fitted distribution has parameters


\begin{gather*} 
	\hat{\bfpi}=\left(
	0.2326, \, 0.4098, \, 0.3576\right)\,, \\ 
	\hat{\bfT}=\left( \begin{array}{ccc}
	-1.3436  & 1.2722 & 0  \\
	0 & -11.2165  & 7.2335  \\
	0 & 0 & -0.8981 \\
	\end{array} \right) \,, \\ 
	\hat{\alpha} = 0.7932  \,.
\end{gather*}

Figure~\ref{fig:mweibull} shows that we recover the shape of the original distribution. Moreover, the parameter $\hat{\alpha} = 0.7932$, which determines the heaviness of the tail, is close to the corresponding one of $\beta = 0.75$ for the given theoretical model. This is no surprise since stable mixing generates the same class as matrix-Weibull distributions. In fact, $-(-\hat{\mat{T}})^\beta$ approximates $\mat{S}$, and although the former matrix is a valid sub-intensity matrix, in the present case it falls outside of the Coxian structure (but it is still upper triangular).
\begin{figure}[h]
\centering
\includegraphics[width=0.49\textwidth]{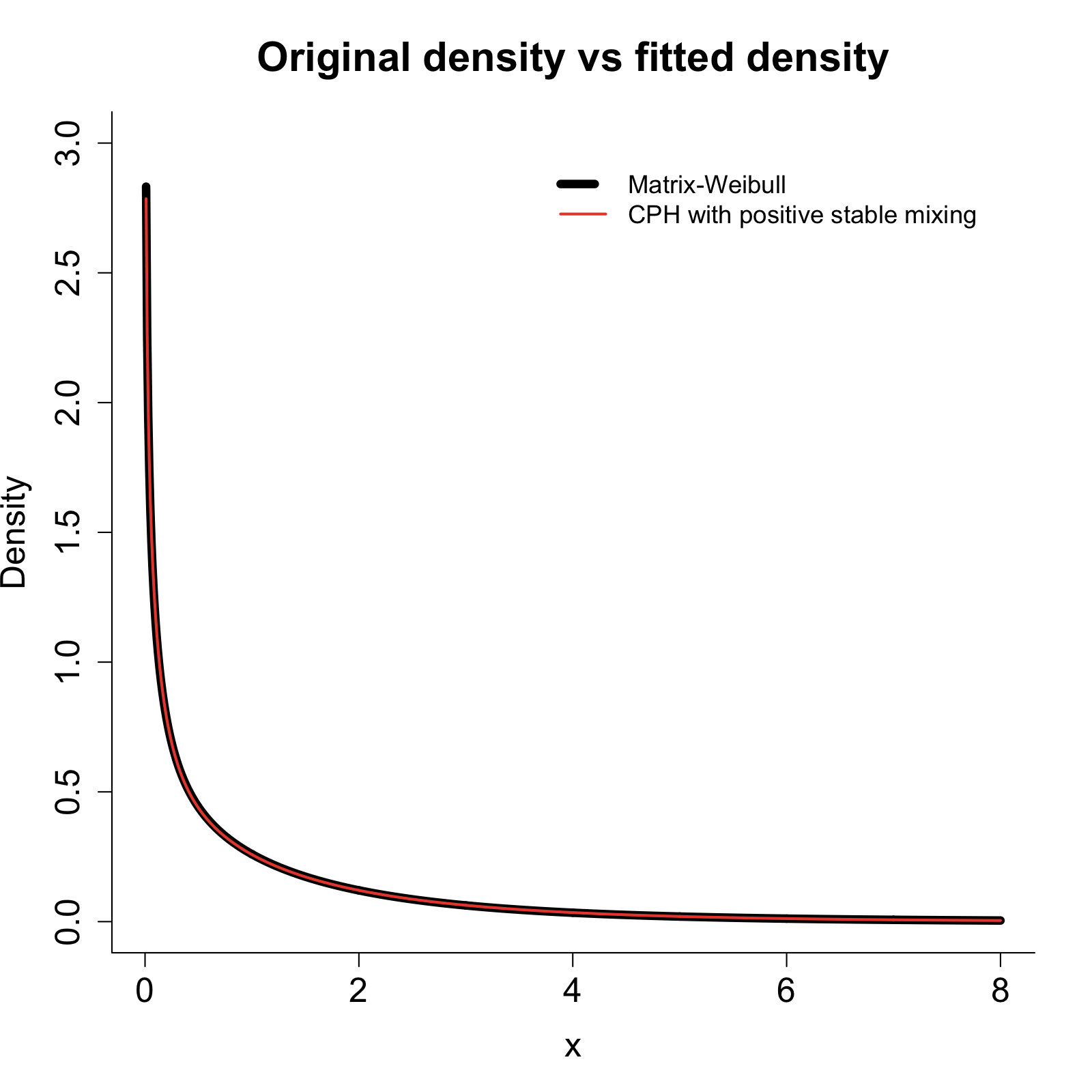}
\includegraphics[width=0.49\textwidth]{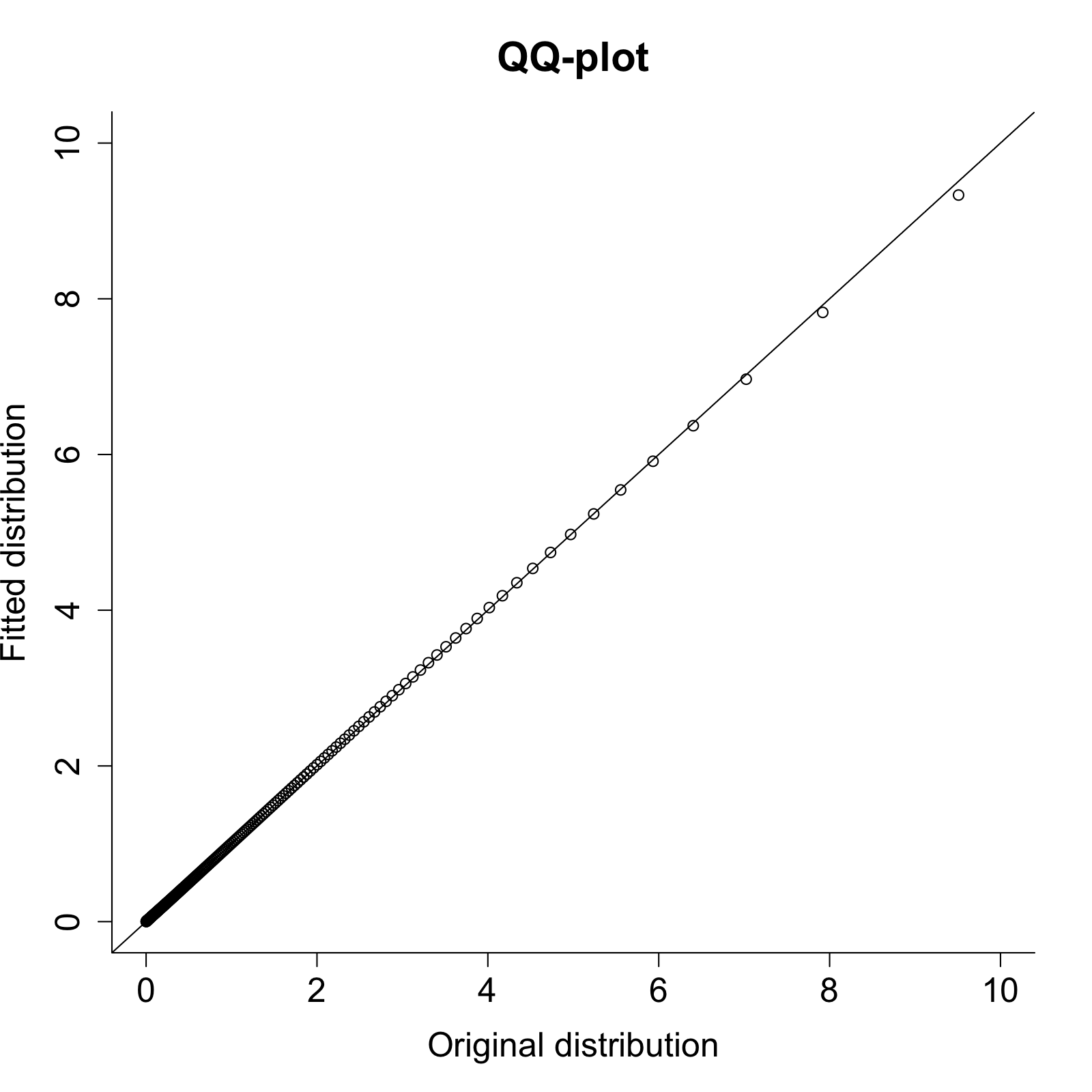}
\caption{Density of the original matrix-Weibull versus density of the fitted CPH (left), and corresponding QQ-plot (right)}
\label{fig:mweibull}
\end{figure}


\section{Conclusion}\label{sec:conclusion}
We studied scaled phase-type distributions when the scaling component is continuous. Particular emphasis was given to the closed-form expressions involved in different functionals, and we adapted and extended some recent results regarding the tail behavior of these distributions. We derived an EM algorithm for maximum likelihood estimation for fully observed data and outlined extensions for censored data. The case of fitting to a given theoretical distribution function was illustrated using a modified version of the latter algorithms. The performance of the proposed algorithm was exemplified in various numerical examples. The results suggest that these global models correctly capture the data's overall shape while remaining in agreement with existing tail behavior analysis. In addition to the results and insight gained for the one-dimensional setup considered in this paper, which is of interest on its own, one can also view the present analysis as a starting point for a multivariate framework when scaling a multivariate random vector with respect to the same continuous random variable. It will be interesting to develop the corresponding implementation and examine the performance on data as well as interpret the implied near-Archimedean dependence structure between the different dimensions from an applied perspective. That approach has potential for parsimonious multivariate modeling, but its concrete development leads to several challenges along the way, which we intend to address in future work.

\textbf{Acknowledgement.} Hansj\"org Albrecher and Martin Bladt would like to acknowledge financial support from the Swiss National Science Foundation Project 200021\_191984.

Jorge Yslas would like to acknowledge financial support from the Swiss National Science Foundation Project IZHRZ0\_180549.

\begin{appendices}\label{ap:def}

 \section{Definitions} 
 \begin{definition} \rm 
 A distribution function $F$ on $\R_{+}=[0, \infty)$ is called: \
	\begin{enumerate}
	\item  {\em Extended regularly varying} if 
		\begin{align*}
			\liminf_{x \to \infty} \frac{\ov F(\lambda x)}{\ov F( x)} \geq \lambda^{-\phi} 
		\end{align*}
	for some $\phi \geq 0$ and all $\lambda\geq1$. The class of extended regularly varying distributions is denoted by $\mathcal{E}$.
	\item {\em Intermediate regularly varying} if 
		\begin{align*}
		\lim_{\lambda \downarrow 1} \liminf_{x \to \infty} \frac{\ov{F}(\lambda x)}{\ov{F}( x)} =1.
		\end{align*}
	 The class of intermediate regularly varying distributions is denoted by $\mathcal{I}$.
	 \item {\em Dominated varying} if 
		\begin{align*}
		\liminf_{x \to \infty} \frac{\ov{F}(\lambda x)}{\ov{F}( x)} > 0
		\end{align*}
	 for some $\lambda>1$. The class of dominated varying distributions is denoted by $\mathcal{D}$.
	 \item {\em Long-tailed} if 
		\begin{align*}
			\lim_{x \to \infty} \frac{\ov{F}(x - y)}{\ov{F}( x)} =1 
		\end{align*}
	for any $y \in \mathbb{R}$. The class of long tailed distributions is denoted by $\mathcal{L}$.
	\end{enumerate}
\end{definition}
 	
 \end{appendices}

\bibliographystyle{apalike}
\bibliography{CPH}

\begin{thebibliography}{}

\bibitem[Albrecher et~al., 2017]{albrecher2017reinsurance}
Albrecher, H., Beirlant, J., and Teugels, J.~L. (2017).
\newblock {\em Reinsurance: Actuarial and Statistical Aspects}.
\newblock John Wiley \& Sons, Chichester.

\bibitem[Albrecher and Bladt, 2019]{albrecher2019inhomogeneous}
Albrecher, H. and Bladt, M. (2019).
\newblock Inhomogeneous phase-type distributions and heavy tails.
\newblock {\em Journal of Applied Probability}, 56(4):1044--1064.

\bibitem[Albrecher et~al., 2020]{albrecher2019matrix}
Albrecher, H., Bladt, M., and Bladt, M. (2020).
\newblock Matrix {M}ittag--{L}effler distributions and modeling heavy-tailed
  risks.
\newblock {\em Extremes}, 23:425--450.

\bibitem[Arendarczyk et~al., 2011]{arendarczyk2011asymptotics}
Arendarczyk, M., Debicki, K., et~al. (2011).
\newblock Asymptotics of supremum distribution of a {G}aussian process over a
  {W}eibullian time.
\newblock {\em Bernoulli}, 17(1):194--210.

\bibitem[Asmussen et~al., 1996]{asmussen1996fitting}
Asmussen, S., Nerman, O., and Olsson, M. (1996).
\newblock Fitting phase-type distributions via the {EM} algorithm.
\newblock {\em Scandinavian {J}ournal of {S}tatistics}, 23(4):419--441.

\bibitem[Berg et~al., 1993]{berg1993generation}
Berg, C., Boyadzhiev, K., and Delaubenfels, R. (1993).
\newblock Generation of generators of holomorphic semigroups.
\newblock {\em Journal of the Australian Mathematical Society}, 55(2):246--269.

\bibitem[Bladt and Nielsen, 2017]{Bladt2017}
Bladt, M. and Nielsen, B.~F. (2017).
\newblock {\em Matrix-Exponential Distributions in Applied Probability}.
\newblock Springer.

\bibitem[Bladt et~al., 2015]{bladt2015calculation}
Bladt, M., Nielsen, B.~F., and Samorodnitsky, G. (2015).
\newblock Calculation of ruin probabilities for a dense class of heavy tailed
  distributions.
\newblock {\em Scandinavian Actuarial Journal}, 2015(7):573--591.

\bibitem[Bladt and Rojas-Nandayapa, 2018]{bladt2017fitting}
Bladt, M. and Rojas-Nandayapa, L. (2018).
\newblock Fitting phase--type scale mixtures to heavy--tailed data and
  distributions.
\newblock {\em Extremes}, 21:285--313.

\bibitem[Bladt and Yslas, 2021a]{bladt2021matrixdist}
Bladt, M. and Yslas, J. (2021a).
\newblock matrixdist: An {R} package for inhomogeneous phase-type
  distributions.
\newblock {\em arXiv preprint arXiv:2101.07987}.

\bibitem[Bladt and Yslas, 2021b]{matrixdist}
Bladt, M. and Yslas, J. (2021b).
\newblock {\em matrixdist: Statistics for Matrix Distributions}.
\newblock R package version 1.1.2.

\bibitem[Breiman, 1965]{breiman1965some}
Breiman, L. (1965).
\newblock On some limit theorems similar to the arc-sin law.
\newblock {\em Theory of Probability \& Its Applications}, 10(2):323--331.

\bibitem[Cline, 1986]{cline1986convolution}
Cline, D.~B. (1986).
\newblock Convolution tails, product tails and domains of attraction.
\newblock {\em Probability Theory and Related Fields}, 72(4):529--557.

\bibitem[Cline, 1987]{cline1987convolutions}
Cline, D.~B. (1987).
\newblock Convolutions of distributions with exponential and subexponential
  tails.
\newblock {\em Journal of the Australian Mathematical Society (Series A)},
  43(03):347--365.

\bibitem[Cline and Samorodnitsky, 1994]{cline1994subexponentiality}
Cline, D.~B. and Samorodnitsky, G. (1994).
\newblock Subexponentiality of the product of independent random variables.
\newblock {\em Stochastic Processes and their Applications}, 49(1):75--98.

\bibitem[Cossette et~al., 2021]{cossette2021univariate}
Cossette, H., Marceau, E., Mtalai, I., and Veilleux, D. (2021).
\newblock Univariate and multivariate mixtures of exponential distributions,
  with applications in risk modeling.
\newblock {\em Applied Stochastic Models in Business and Industry}.

\bibitem[Doolittle, 1998]{doolittle1998analytic}
Doolittle, E. (1998).
\newblock Analytic functions of matrices.
\newblock {\em Lecture Note}.

\bibitem[Embrechts and Goldie, 1980]{embrechts1980closure}
Embrechts, P. and Goldie, C.~M. (1980).
\newblock On closure and factorization properties of subexponential and related
  distributions.
\newblock {\em Journal of the Australian Mathematical Society (Series A)},
  29(02):243--256.

\bibitem[Embrechts et~al., 2013]{embrechts2013modelling}
Embrechts, P., Kl{\"u}ppelberg, C., and Mikosch, T. (2013).
\newblock {\em Modelling Extremal Events: For Insurance and Finance},
  volume~33.
\newblock Springer Science \& Business Media.

\bibitem[Foss et~al., 2011]{foss2011introduction}
Foss, S., Korshunov, D., Zachary, S., et~al. (2011).
\newblock {\em An Introduction to Heavy-Tailed and Subexponential
  Distributions}, volume~6.
\newblock Springer.

\bibitem[Frees and Valdez, 1998]{frees1998understanding}
Frees, E.~W. and Valdez, E.~A. (1998).
\newblock Understanding relationships using copulas.
\newblock {\em North American Actuarial Journal}, 2(1):1--25.

\bibitem[Furman et~al., 2021]{furman2021}
Furman, E., Kye, Y., and Su, J. (2021).
\newblock Multiplicative background risk models: Setting a course for the
  idiosyncratic risk factors distributed phase-type.
\newblock {\em Insurance: Mathematics and Economics}, 96:153--167.

\bibitem[Higham, 2008]{Higham2008}
Higham, N.~J. (2008).
\newblock {\em Functions of Matrices: Theory and computation}.
\newblock Society for Industrial and Applied Mathematics (SIAM), Philadelphia,
  PA.

\bibitem[Jacobsen et~al., 2009]{jacobsen2009inverse}
Jacobsen, M., Mikosch, T., Rosi{\'n}ski, J., Samorodnitsky, G., et~al. (2009).
\newblock Inverse problems for regular variation of linear filters, a
  cancellation property for $\sigma$-finite measures and identification of
  stable laws.
\newblock {\em The Annals of Applied Probability}, 19(1):210--242.

\bibitem[Maulik and Resnick, 2004]{maulik2004characterizations}
Maulik, K. and Resnick, S. (2004).
\newblock Characterizations and examples of hidden regular variation.
\newblock {\em Extremes}, 7(1):31--67.

\bibitem[Olsson, 1996]{olsson1996estimation}
Olsson, M. (1996).
\newblock Estimation of phase-type distributions from censored data.
\newblock {\em Scandinavian Journal of Statistics}, 23(4):443--460.

\bibitem[Rojas-Nandayapa and Xie, 2018]{rojas2018asymptotic}
Rojas-Nandayapa, L. and Xie, W. (2018).
\newblock Asymptotic tail behaviour of phase-type scale mixture distributions.
\newblock {\em Annals of Actuarial Science}, 12(2):412--432.

\bibitem[Rudin, 1987]{rudin87}
Rudin, W. (1987).
\newblock {\em Real and Complex Analysis, 3rd Ed.}
\newblock McGraw-Hill, Inc., USA.

\bibitem[Samorodnitsky and Sun, 2016]{samorodnitsky2016multivariate}
Samorodnitsky, G. and Sun, J. (2016).
\newblock Multivariate subexponential distributions and their applications.
\newblock {\em Extremes}, 19(2):171--196.

\bibitem[Schilling et~al., 2012]{schilling2012bernstein}
Schilling, R.~L., Song, R., and Vondracek, Z. (2012).
\newblock {\em Bernstein functions}.
\newblock de Gruyter.

\bibitem[Tang, 2006]{tang2006subexponentiality}
Tang, Q. (2006).
\newblock The subexponentiality of products revisited.
\newblock {\em Extremes}, 9(3-4):231--241.

\bibitem[Xu et~al., 2017]{xu2017necessary}
Xu, H., Cheng, F., Wang, Y., and Cheng, D. (2017).
\newblock A necessary and sufficient condition for the subexponentiality of the
  product convolution.
\newblock {\em Advances in Applied Probability}, 50(1):57--73.

\end{thebibliography}

\end{document}